\documentclass[10pt]{amsart}

\usepackage{hyperref}
\usepackage{enumerate}
\usepackage{comment}

\makeatletter

\@namedef{subjclassname@2010}{

  \textup{2010} Mathematics Subject Classification}

\usepackage{amssymb, amsmath,amsthm}
\usepackage{mathrsfs}
\newtheorem{thm}{Theorem}[section]

\newtheorem{prop}[thm]{Proposition}

\newtheorem{cor}[thm]{Corollary}

\newtheorem{lem}[thm]{Lemma}
\theoremstyle{definition}
\newtheorem{rem}{Remark}

\newcommand{\ra}{\rightarrow}

\newcommand{\mc}{\mathcal}

\newcommand{\mb}{\mathbb}
\newcommand{\sg}{\sigma}

\newcommand{\Z}{\mb{Z}}
\renewcommand{\ss}{\substack}

\newcommand{\e}{\varepsilon}

\renewcommand{\bar}{\overline}

\begin{document}
\title{Large sums of high order characters II}
\author{Alexander P. Mangerel and Yichen You}
\address{Department of Mathematical Sciences, Durham University, Stockton Road, Durham, DH1 3LE, UK}
\email{smangerel@gmail.com \\ yichen.you@outlook.com}
\begin{abstract}
Let $\chi$ be a primitive character modulo $q$, and let $\delta > 0$. Assuming that $\chi$ has large order $d$, for any $d$th root of unity $\alpha$ we obtain non-trivial upper bounds for the number of $n \leq x$ such that $\chi(n) = \alpha$, provided $x > q^{\delta}$. This improves upon a previous result of the first author by removing restrictions on $q$ and $d$. As a corollary, we deduce that if the largest prime factor of $d$ satisfies $P^+(d) \to \infty$ then the level set $\chi(n) = \alpha$ has $o(x)$ such solutions whenever $x > q^{\delta}$, for any fixed $\delta > 0$. 

  Our proof relies, among other things, on a refinement of a mean-squared estimate for short sums of the characters $\chi^\ell$, averaged over $1 \leq \ell \leq d-1$, due to the first author, which goes beyond Burgess' theorem as soon as $d$ is sufficiently large. We in fact show the alternative result that either (a) the partial sum of $\chi$ itself, or (b) the partial sum of $\chi^\ell$, for ``almost all'' $1 \leq \ell \leq d-1$, exhibits cancellation on the interval $[1,q^{\delta}]$, for any fixed $\delta > 0$. 
  
  By an analogous method, we also show that the P\'{o}lya-Vinogradov inequality may be improved for either $\chi$ itself or for almost all $\chi^\ell$, with $1 \leq \ell \leq d-1$. In particular, our averaged estimates are non-trivial whenever $\chi$ has sufficiently large \emph{even} order $d$.
\end{abstract}
\maketitle

\section{Introduction and main results}
The objective of this paper is to improve the results of \cite{ManHighOrd} on averages of short and maximal sums of a Dirichlet character whose (group-theoretic) order is large. In \cite{ManHighOrd}, the first author considered primitive  Dirichlet characters $\chi$ modulo a prime $q$ with order $d$, under the assumption that $d = d(q) \ra \infty$ as $q \ra \infty$.
In that work the first author investigated how this assumption on $d$ influenced the sizes of the short sums
$$
S_{\chi^\ell}(x) := \sum_{n \leq x} \chi^{\ell}(n), \quad x > q^{\delta}
$$
for arbitrary fixed $\delta > 0$, and the maximal sums
$$
M(\chi^\ell) := \max_{1 \leq t \leq q} \left|\sum_{n \leq t} \chi^{\ell}(n)\right|
$$
for $1 \leq \ell \leq d-1$. The methods of \cite{ManHighOrd} had the defect that they only yielded non-trivial results under the assumption that the least prime factor of $d$ was also assumed to be large. 

In this paper, we rectify this shortcoming by presenting (quantitatively stronger) analogues of the theorems in \cite{ManHighOrd} in which assumptions on the size of the prime factors of $d$ are removed. Moreover, the results in this paper apply to general moduli $q$, rather than just to prime $q$.\\
Our first main theorem is an alternative bound, which states that in the r\'{e}gime that $d \ra \infty$ with $q$, either $|S_{\chi}(x)| = o(x)$ or else the mean-square average of the short character sums $S_{\chi^\ell}(x)$ with length $x = q^{\delta}$ exhibits cancellation. 
\begin{thm}\label{thm:impShort}
Let $q \geq 3$ and let $\chi$ be a primitive Dirichlet character modulo $q$ with order $d \geq 2$. Then there is an absolute constant $c > 0$ such that if $\tau \in (0,1/2)$,
$$
\delta := \max\left\{\left(\frac{\log\log(ed)}{c\log (ed)}\right)^{1/2}, (\log q)^{-c}\right\},
$$
and $x > q^{\delta}$ then at least one of the following is true:
\begin{enumerate}
\item $\chi$ itself satisfies
$$
\frac{1}{x}\left|\sum_{n \leq x} \chi(n)\right| \ll_{\tau} \frac{1}{(\log\log (ed))^{1/6-\tau}};
$$
\item we have the average bound
$$
\frac{1}{d}\sum_{1 \leq \ell \leq d} \left|\frac{1}{x}\sum_{n \leq x} \chi^\ell(n)\right|^2 \ll_{\tau} \frac{1}{(\log\log (ed)^{1/6-\tau}}.
$$
\end{enumerate}
\end{thm}
This should be compared with \cite[Thm. 2]{ManHighOrd}, in which the averaged bound was only non-trivial under the assumption that $P^-(d) \ra \infty$ with $d$, and the savings only comparable if $P^-(d) \gg \log\log d$. \\
As in \cite{ManHighOrd}, a mean-square bound like Theorem \ref{thm:impShort} (together with some additional inputs) may be used to prove a paucity phenomenon for the level sets of $\chi$. In this direction, our second main theorem provides a non-trivial upper bound for the cardinality of the set of solutions $n \leq x$ with $\chi(n) = \alpha$, for any fixed $d$th order root of unity $\alpha$, whenever $d \ra \infty$ and $x > q^\delta$ for $\delta > 0$ fixed but arbitrary. This strictly generalises \cite[Thm. 1]{ManHighOrd}, wherein the condition that $d$ be squarefree had to be assumed. 
\begin{thm} \label{thm:paucity}
There are absolute constant $c_1,c_2 > 0$ such that the following holds. \\
Let $q \geq 3$ and let $\chi$ be a primitive Dirichlet character modulo $q$ with order $d \geq 2$. For each $z \geq 1$, define
$$
d_z := \prod_{\ss{p^k || d \\ p > z}} p^k, \quad
\delta_z := \max\left\{\left(\frac{\log\log(ed_z)}{c_1\log (ed_z)}\right)^{1/2}, (\log q)^{-c_1}\right\}.
$$
Then if $x > q^{\delta_1}$,
\begin{equation} \label{eq:infBdpau}
\max_{\alpha^d = 1} \frac{1}{x}|\{n \leq x : \chi(n) = \alpha\}| \leq \inf_{\ss{1 \leq z \leq \log\log (ed) \\ x > q^{\delta_z}}} \left(\frac{1}{z} + O\left(\frac{1}{(\log\log (ed_z))^{c_2}}\right)\right).
\end{equation}
\end{thm}
\begin{rem}
Let us show that Theorem \ref{thm:paucity} indeed generalises \cite[Thm. 1]{ManHighOrd}. If $d$ is squarefree and $z = \log\log(ed)$ then by the prime number theorem,
$$
d/d_z \leq \prod_{p \leq z} p \leq e^{2z} \asymp (\log d)^2,
$$
provided $d$ is sufficiently large. The upper bound from Theorem \ref{thm:paucity} is thus of quality $O(1/(\log\log (ed))^{c_2})$, which is comparable (albeit with a less explicit $\log\log(ed)$ power) with \cite[Thm. 1]{ManHighOrd}. 
\end{rem}
\begin{rem}
To explain the form of the upper bound given in Theorem \ref{thm:paucity} it is helpful to consider a case, not covered in \cite[Thm. 1]{ManHighOrd}, where $d$ is a prime power, say $d = 2^k$. In this case, $d_2 = 1$, so that, taking $z \ra 2^-$, the upper bound provided by Theorem \ref{thm:paucity} is precisely
\begin{equation}\label{eq:pow2}
\max_{\alpha^{2^k} = 1} \frac{1}{x}|\{n \leq x : \chi(n) = \alpha\}| \leq \frac{1}{2} + o_{k \ra \infty}(1)
\end{equation}
(and indeed this is the worst-possible bound that \eqref{eq:infBdpau} provides in general). It can be shown that Theorem \ref{thm:impShort} implies the bound
$$
\max_{\alpha^{2^k} = 1} \frac{1}{x}|\{n \leq x : \chi(n) = \alpha\}| \ll \frac{1}{(\log k)^{1/13}} \text{ whenever } |S_{\chi}(x)| > \frac{x}{(\log k)^{2/13}},
$$
which is of course much stronger as $k \ra \infty$. In the converse case that $|S_{\chi}(x)|$ is small, however, the following heuristically plausible scenario is consistent with \eqref{eq:pow2}. \\
Suppose that $\chi$ has order $2^k$, but satisfies $\chi(p) = \pm 1$ for all $p \leq x$ and
$$
\sum_{\ss{p \leq x \\ \chi(p) = -1}} \frac{1}{p} \ra \infty \text{ as } k \ra \infty.
$$
Thus, $\chi$ is a real-valued multiplicative function on $[1,x]$. By a theorem of Hall and Tenenbaum \cite{HT}, we obtain
$$
\left|S_{\chi}(x)\right| \ll x \exp\left(-\frac{1}{4} \sum_{p \leq x} \frac{1-\chi(p)}{p}\right) = o_{k \ra \infty}(x).
$$
Since $\chi(n) \in \{-1,+1\}$ for all $n \leq x$, this is equivalent to
$$
\max_{\alpha \in \{-1,+1\}} \frac{1}{x}|\{n \leq x : \chi(n) = \alpha\}| = \frac{1}{2} + o_{k \ra \infty}(1),
$$
which is precisely of the form \eqref{eq:pow2}.
\end{rem}
We obtain the following straightforward consequence of Theorem \ref{thm:paucity}.
\begin{cor} \label{cor:paucity}
Assume the notation and hypotheses of Theorem \ref{thm:paucity}, and let
$$
\delta := \max\left\{\left(\frac{\log\log(eP^+(d))}{c_1\log(eP^+(d))}\right)^{1/2}, (\log q)^{-c_1}\right\}.
$$
Then if $x > q^{\delta}$ we get
$$
\max_{\alpha^d = 1} \frac{1}{x} |\{n \leq x : \chi(n) = \alpha\}| \ll \frac{1}{(\log\log (eP^+(d)))^{c_2}}. 
$$
\end{cor}
Corollary \ref{cor:paucity} shows that as long as $P^+(d) \ra \infty$ with $d$, the level sets $|\{n \leq x : \chi(n) = \alpha\}|$ are sparse as soon as $x > q^{\delta}$, for any fixed, but otherwise arbitrary, $\delta > 0$.  Note that this property is fairly generic, only excluding orders $d$ that are very smooth (and hence rare).\\
In \cite{ManHighOrd} the first author also gave a non-trivial average bound for the maximal character sums $M(\chi^\ell)$, $1 \leq \ell \leq d-1$.
The P\'{o}lya-Vinogradov inequality states that for a non-principal character $\psi$ of modulus $m$ we have $M(\psi) \ll \sqrt{m} \log m$. It is a long-standing open problem to obtain unconditional improvements (as $m \ra \infty$) to this bound for general $\psi$. In \cite{ManHighOrd} (see Theorem 3 there), the bound
$$
\frac{1}{d}\sum_{1 \leq \ell \leq d-1} M(\chi^\ell) \ll \left(\sqrt{q}\log q\right) \left(\frac{1}{P^-(d)} + \sqrt{\frac{\log\log \log q}{\log \log q}}\right)
$$
was obtained by appealing to combinatorial arguments. Clearly, this bound is non-trivial only when $d$ has no small prime factors and therefore must be \emph{odd}. Well-known work of Granville and Soundararajan \cite{GSPret} (with refinements in \cite{Gold} and \cite{LamMan}) previously showed that $M(\chi) = o(\sqrt{q}\log q)$ whenever $\chi$ has \emph{odd} order $d = o(\sqrt{\log\log q})$, so that this result is \emph{only new} when 
$$
d \gg \sqrt{\log\log q} \text{ and } P^-(d) \ra \infty \text{ as } d \ra \infty. 
$$
Our next theorem remedies this situation, providing non-trivial bounds as soon as $d \ra \infty$ (including the case that $d$ is even).
\begin{thm} \label{thm:impMax}
Let $q \geq 3$ and let $\chi$ be a primitive Dirichlet character modulo $q$ with order $d \geq 2$. Then at least one of the following statements is true:
\begin{enumerate}[(i)]
\item $\chi$ itself satisfies
$$
M(\chi) \ll \frac{\sqrt{q}\log q}{(\log\log (ed))^{1/8}}
$$
\item we have
$$
\frac{1}{d}\sum_{1 \leq \ell \leq d-1} M(\chi^\ell) \ll \frac{\sqrt{q}\log q}{(\log\log(ed))^{1/8}}.
$$
\end{enumerate}
\end{thm}
\subsection{Proof strategy}
Let us describe separately the strategy of proof of each of the main theorems in this paper.
\subsubsection{On averages of short character sums}
The proof of Theorem \ref{thm:impShort} largely follows the line of attack of \cite{ManHighOrd}, introducing refinements of the key lemmas at several junctures. \\
Ideally, we would like to prove that $|S_{\chi}(x)| = o_{d \ra \infty}(x)$ when $x > q^{\delta}$ for $\delta \in (0,1)$, and the first alternative of Theorem \ref{thm:impShort} is consistent with this goal. We shall mainly discuss the consequences of assuming that this alternative in fact fails.\\
As in \cite{ManHighOrd}, given small parameters $\delta,\e \in (0,1)$ we study the structure of the ``large spectrum'' set
$$
\mc{C}_d(\e) := \{1 \leq \ell \leq d-1 : |S_{\chi^\ell}(x)| \geq \e x\}, \quad x > q^\delta.
$$
If $|\mc{C}_d(\e)| \leq \e d$ then the $L^2$ average of $|S_{\chi^\ell}(x)|$ is $\ll \e$. Our main objective is to prove that this upper bound on $|\mc{C}_d(\e)|$ indeed holds. \\
Suppose instead that $|\mc{C}_d(\e)| > \e d$. In this case we use results from additive combinatorics to derive a structure theorem for $\mc{C}_d(\e)$ (see Proposition \ref{prop:shortStruc}). Precisely, we show that there is $m = O_{\e}(1)$ such that the $m$-fold sumset
$$
m\mc{C}_d(\e) := \{a_1 + \cdots + a_m \pmod{d} : a_j \in \mc{C}_d(\e) \text{ for all } 1 \leq j \leq m\}
$$
coincides with a large subgroup $H \leq \mb{Z}/d\mb{Z}$, and that for each $\ell \in H$ the character $\chi^{\ell}$ ``pretends to be'' an archimedean character $n^{it_{\ell}}$, with $\max_{\ell \in H} |t_{\ell}|\log x = O_{\e}(1)$. As a consequence, we deduce that
\begin{equation} \label{eq:distboundIntro}
\sum_{p \leq x} \frac{1-\text{Re}(\chi(p)^\ell)}{p} = O_{\e}(1) \text{ uniformly over } \ell \in H.
\end{equation}
Unlike in \cite{ManHighOrd} where $P^-(d)$ was assumed to be large, here the subgroup $H$ need not be the entirety of $\mb{Z}/d\mb{Z}$. Nevertheless, the fact that $|H| \gg_{\e} d$ is what is crucial in the forthcoming analysis. \\
As in \cite{ManHighOrd}, the argument then splits according to the nature of the prime level sets\footnote{As usual, given $t \in \mb{R}$ we write $e(t) := e^{2\pi i t}$.}
$$
S_j := \{p \leq x : \chi(p) = e(j/d)\}, \quad 1 \leq j \leq d-1,
$$
and in particular the associated reciprocal sums
$$
\sg_j(x) := \sum_{\ss{p \leq x \\ p \in S_j}} \frac{1}{p}, \quad 1 \leq j \leq d-1.
$$
After showing that 
$$
\Sigma_{\chi}(x) := \sum_{\ss{p \leq x \\ \chi(p) \neq 0,1}} \frac{1}{p} = \sum_{1 \leq j \leq d-1} \sg_j(x) \ra \infty \text{ as } d \ra \infty
$$
(see Proposition \ref{prop:LowBdSigma}, which is a very slight generalisation of \cite[Thm. 1.1]{ManHighOrd} to composite moduli $q$), we consider two cases. First, if $\max_{1 \leq j \leq d-1} \sg_j(x)$ is rather small compared to $\Sigma_{\chi}(x)$ then we show that there is $\ell \in H$ such that \eqref{eq:distboundIntro} cannot hold (see Lemma \ref{prop:smallSig}). This follows the lines of \cite[Lem. 4.4]{ManHighOrd}. Namely, having first observed that
$$
\sum_{p \leq x} \frac{1-\text{Re}(\chi^\ell(p))}{p} \geq 8 \sum_{1 \leq j \leq d-1} \|\frac{j\ell}{d}\|^2 \sg_j(x),
$$
we use Fourier analysis to obtain a lower bound for the left-hand side sum for \emph{some} $\ell \in H$ by showing a variance bound of the shape\footnote{Given $t \in \mb{R}$ we write $\|t\| := \min\{\{t\}, 1-\{t\}\}$.}
$$
\frac{1}{|H|} \sum_{\ell \in H} \left(\sum_{1 \leq j \leq d-1} \left\|\frac{j\ell}{d} \right\|^2 \sg_j(x) - \frac{1}{12}\Sigma_{\chi}(x)\right)^2 = o_{d \ra \infty}(\Sigma_{\chi}(x)^2).
$$
Whereas the argument in \cite[Lem. 4.4]{ManHighOrd} made use of the fact that $P^-(d)$ was large, we manage to circumvent this assumption by a more careful argument. \\
In the case that $\sg_{j_0}(x) := \max_{1 \leq j \leq d-1} \sg_j(x) \gg \Sigma_{\chi}(x)$ we provide a quantitatively stronger variant of \cite[Prop. 4.5]{ManHighOrd}. The idea there was to establish an asymptotic of the shape
\begin{equation} \label{eq:comptwistedSchi}
\sum_{n \leq x} \chi^\ell(n) =  e(j_0 \ell/d)\sum_{n \leq x} \chi^\ell(n) + o_{d \ra \infty}(x),
\end{equation}
by using the Tur\'{a}n-Kubilius inequality\footnote{A similar application of this idea will be discussed in Section \ref{subsubsec:max} below}  to show that most integers $n \leq x$ have $\sim \sg_{j_0}(x)$ prime divisors $p \in S_{j_0}$. For each of these prime divisors, if $n = mp$ then $\chi^\ell(n) = e(j_0\ell/d) \chi^\ell(m)$, and using Lipschitz estimates for multiplicative functions the partial sum $S_{\chi^\ell}(x)$ for $n \leq x$ can be well-approximated by the sum $S_{\chi^\ell}(x/p)$ for $m\leq x/p$ (as long as $p$ is not too large). \\
Our refinement of this idea, found in Proposition \ref{prop:largekBd} below, generalises this from single primes $p \in S_{j_0}$ to products of $k$ prime factors $p \in S_{j_0}$, where $k = o(\sqrt{\sg_{j_0}(x)})$. The flexibility in the choice of $k$ is what is ultimately responsible for the improved exponent of $\log\log d$ in Theorem \ref{thm:impShort}, relative to \cite[Thm. 2]{ManHighOrd}. \\
Note that \eqref{eq:comptwistedSchi} is only useful in proving $|S_{\chi^\ell}(x)| = o_{d \ra \infty}(x)$ provided that 
$$
|1-e(j_0\ell/d)| \asymp \|j_0 \ell/d\| \gg 1.
$$ 
In \cite{ManHighOrd} the condition $P^-(d) \ra \infty$ proved advantageous in showing that this was the case for \emph{most} $\ell \in H$. 
Indeed, since $1 \leq j_0,\ell < d$, we have $\gamma := (j_0,d) \leq d/P^-(d)$. Setting $a = j_0/\gamma$ and $d' = d/\gamma$, it follows that
$$
\left\|\frac{j_0\ell}{d} \right\| = \left\|\frac{a\ell}{d'} \right \|,
$$
and it can be shown that when $d'$ is large, \emph{most} choices of $\ell$ satisfy $\|j_0\ell/d\| \gg_{\e} 1$, essentially because the range $[\e d', (1-\e) d']$, say, is large.\\ 
This certainly fails if $P^-(d)$ is small. For instance, if $d$ is even then it is plausible that $j_0 = d/2$, and so $\|j_0\ell/d\| = 0$ for approximately half of all $1 \leq \ell \leq d-1$. The issue here is that $\gamma = (j_0,d)$ is excessively large, i.e., of size $\gg d$, in this case. On the other hand, we show that $\gamma$ is rather smaller than $d$ whenever $|S_{\chi}(x)|$ is large, and in this case we may again conclude that $\|j_0\ell/d\| \gg_{\e} 1$ for most $1 \leq \ell \leq d-1$. This is precisely the reason for assuming that $|S_{\chi}(x)|$ is large in the second alternative in Theorem \ref{thm:impShort}. 
\subsubsection{A new bound for level sets of $\chi$}
To prove Theorem \ref{thm:paucity} we employ three observations (see Lemmas \ref{lem:passtoPow} and \ref{lem:ETL2Alt} below), two of which are already present in \cite{ManHighOrd}. Firstly, if $b|d$, $\alpha$ is a $d$th order root of unity and $\beta := \alpha^b$ then we have the trivial inclusion
$$
\{n \leq x : \chi(n) = \alpha\} \subseteq \{n \leq x: \chi^b(n) = \beta\}.
$$
This allows us to replace a  bound for the level sets of $\chi$ of order $d$ by those of $\psi := \chi^b$ of some order $d' = d/b$ dividing $d$. As discussed below, this sometimes presents an advantage. \\
Secondly, the level sets of $\chi$ can be linked to the $L^2$ averages of the powers $\chi^\ell$. This follows by orthogonality modulo $d$ from the formula
$$
S_{\chi^\ell}(x) = \sum_{n \leq x} \chi^\ell(n) = \sum_{\alpha^d = 1} \alpha^\ell |\{n \leq x : \chi(n) = \alpha\}|.
$$
Therefore, whenever a non-trivial bound is available for the $L^2$ average
\begin{equation} \label{eq:L2Intro}
\frac{1}{d} \sum_{1 \leq \ell \leq d} |S_{\chi^\ell}(x)|^2,
\end{equation}
we obtain correspondingly non-trivial bounds for all level sets $|\{n\leq x : \chi(n) = \alpha\}|$. By Theorem \ref{thm:impShort} this is true as long as $|S_{\chi}(x)|$ is large. \\
The third key observation concerns the converse case, namely when $|S_{\chi}(x)|$ is small (and therefore non-trivial bounds for \eqref{eq:L2Intro} do not follow from Theorem \ref{thm:impShort}). In this case, we may arrive at a bound for the level sets by interpreting the event $\chi(n) = \alpha = e(a/d)$, $0 \leq a \leq d-1$, as a constraint on the distribution of the complex argument $\theta_n \in [0,1)$ of $\chi(n)$ (provided $(n,q) = 1$), i.e., 
$$
\chi(n) = e(\theta_n) \text{ with } \theta_n \in \left[\frac{a}{d},\frac{a+1}{d}\right).
$$ 
Since this interval has measure $1/d$, if $\theta_n$ were uniformly distributed we would expect the number of such $n \leq x$ to have size $\sim x/d$. Using the Erd\H{o}s-Tur\'{a}n inequality to control the deviation from this heuristic, we show that $|\{n \leq x: \chi(n) = \alpha\}|$ may be bounded above by 
\begin{equation}\label{eq:ETIntro}
\frac{x}{d} + \frac{x}{K+1} + O\left(\sum_{1 \leq k \leq K} \frac{1}{k} |S_{\chi^k}(x)|\right) \text{ for any } K \geq 1.
\end{equation}
Knowing that $|S_{\chi}(x)|$ is small, we use the pretentious theory of multiplicative functions to show (in certain key cases where upper bounds for \eqref{eq:L2Intro} are not available) that in fact  $|S_{\chi^k}(x)|$ is also small whenever $1 \leq k < P^-(d)$. This can be understood as being due to 
$$
\chi^k(p) \neq 1 \text{ whenever } \chi(p) \neq 1 \text{ and } 1 \leq k < P^-(d),
$$ 
so that $\chi^k$ retains much of the oscillation exhibited when $\chi$ has small partial sums. The upshot of this is that we may then select $K = P^-(d)-1$ in \eqref{eq:ETIntro}. This bound presents no advantage when $P^-(d)$ is quite small, but can be strengthened in the case that the contribution to $d$ from its small prime factors is small. More precisely, applying the first observation above with 
$$
d' = \prod_{\ss{p^k||d \\ p > z}} p^k \text{ for any } 1 \leq z \leq \log\log(ed)
$$ 
allows us (after replacing $d$ and $\chi$ by $d'$ and $\chi^{d/d'}$, respectively) to apply \eqref{eq:ETIntro} with $K = P^-(d') -1 > z - 1$ instead. This improves the bound, as long as $d'$ is sufficiently large. 
\subsubsection{On averages of maximal sums} \label{subsubsec:max}
Our expectation is that the first alternative of Theorem \ref{thm:impMax} always holds, but here we will mainly focus on the consequences if it fails. As in \cite{ManHighOrd}, given a small parameter $\e > 0$, we investigate the structure of
\begin{align*}
    \mc{L}_d(\e) := \{1 \leq \ell \leq d-1: |M(\chi^\ell)| \geq \e \sqrt{q} \log q\}.
\end{align*}
If $|\mc{L}_d(\e)| \leq \e d$ then the average size of $M(\chi^\ell)$ is $\ll \e\sqrt{q}\log q$. In other words, for most $1 \leq \ell \leq d-1$, $M(\chi^\ell)$ admits a sharper upper bound than what the Pólya-Vinogradov inequality provides. Our goal is to show that $|\mc{L}_d(\e)|$ is indeed of size $O(\e d)$.

By Proposition \ref{prop:MpsiAsymp} and Lemma \ref{lem:GSLog}, bounding $M(\chi^\ell)$ reduces to the estimation of a logarithmic sum
\begin{align*}
    L_{\chi^\ell\bar{\gamma}_{\ell}}(N) := \sum\limits_{n \leq N}\frac{\chi^\ell\bar{\gamma}_{\ell}(n)}{n},
\end{align*}
where $\gamma_{\ell}$ is some Dirichlet character of small conductor determined by $\chi^\ell$, and $N = N_{\ell} \in [1,q]$. In turn, this can be related via standard estimates for logarithmic averages of multiplicative functions, to the prime sum 
\begin{align*}
    \sum\limits_{p \leq q}\frac{1-\text{Re}(\chi^\ell\bar{\gamma}_{\ell}(p))}{p}.
\end{align*}
In the same vein as the structure theorem for $\mc{C}_d(\e)$, which involved classifying those archimedean characters $n^{it_\ell}$ to which the characters $\chi^\ell$, $\ell \in \mc{C}_d(\e)$ were pretentious, we show a \emph{non-archimedean} analogue of this for $\mc{L}_d(\e)$. Namely, we show that there are $m,g = O_{\e}(1)$ such that $m\mc{L}_d(\e)$ is a subgroup $H = \langle g \rangle \leq \Z/d\Z$, where $|H| = d/g \gg_{\e} d$. Assuming $\chi^g$ ``pretends to be" a primitive Dirichlet character $\xi$ of order $r$, each $g\ell \in H$ the character $\chi^{g\ell}$ also ``pretends to be" $\xi^\ell$ (in a manner that is uniform in $\ell$). As a result, setting $\psi := \chi^g \bar{\xi}$ we show that
\begin{align*}
    \max_{\ell \in H}\sum\limits_{p \leq x}\frac{1-\text{Re}(\psi^\ell(p))}{p} = O_{\e}(1).
\end{align*}
which is of the same shape as \eqref{eq:distboundIntro}. 
We have thus reduced matters in this problem to a situation similar to that of the short character sums problem, replacing $\chi^\ell$ by $\psi^\ell$, for $\ell \in H$. By considering the prime level sets of $\psi$, we can apply analogous arguments to those used in the proof of Theorem \ref{thm:impShort}. \\
More precisely, let $\omega = e(1/r)$ and
$$\tilde{\sigma}_{j_0} := \max\limits_{1 \leq j < r} \sum\limits_{\substack{p \leq q, \\\psi(p) = \omega^j}} \frac{1}{p}.$$
The case when $\tilde{\sg}_{j_0}$ is small relative to $\Sigma_{\psi}(q)$ is completely analogous to the corresponding case in the proof of Theorem \ref{thm:impShort}, and so we focus here on the case that $\tilde{\sg}_{j_0} \gg \Sigma_{\psi}(q)$. We seek to obtain an asymptotic formula of the type in \eqref{eq:comptwistedSchi}, for the logarithmic sums $L_{\psi^\ell}(N_\ell)$. In fact, we prove that
\begin{align}\label{eq:logsum}
     L_{\psi^\ell}(N_\ell) = e(j_0\ell/d)L_{\psi^\ell}(N_\ell) + o\left(\frac{\log N_\ell}{\sqrt{\tilde{\sigma}_{j_0}}}\right), 
\end{align}
 where $|L_{\psi^\ell}(N_\ell)| = \max\limits_{1 \leq N \leq q}|L_{\psi^\ell}(N)|$. 
 The idea is to view $\tilde{\sigma}_{j_0}$ as the average value of the completely additive function
\begin{align*}
\Omega_{j_0}(n) 
=\sum\limits_{\substack{p^k | n, \\ \psi(p) = \omega^{j_0} }} 1,
\end{align*}
and by the Tur\'{a}n-Kubilius inequality we have $\Omega_{j_0}(n) \sim \tilde{\sg}_{j_0}$ for \emph{most} $n \leq q$. In particular, we find that
$$
L_{\psi^\ell}(N_\ell) \sim \frac{1}{\tilde{\sg}_{j_0}(n)} \sum_{n \leq N_\ell} \frac{\psi^\ell(n) \Omega_{j_0}(n)}{n} \sim \frac{1}{\tilde{\sg}_{j_0}} \sum_{\ss{p \leq N_\ell \\ \psi(p) = \omega^{j_0}}} \frac{\psi^\ell(p)}{p} \sum_{m \leq N_\ell/p} \frac{\psi^\ell(m)}{m} = \frac{\omega^{\ell j_0}}{\tilde{\sg}_{j_0}} \sum_{\ss{p \leq N_\ell \\ \psi(p) = \omega^{j_0}}} \frac{1}{p}L_{\psi^\ell}(N_\ell/p).
$$
Using the trivial estimate $L_{\psi^\ell}(N_\ell/p) = L_{\psi^\ell}(N_\ell) + O(\log p)$ and  Mertens' theorem, we arrive at \eqref{eq:logsum}. (While this gives a quantitatively weaker estimate than what might be obtained by the more general, yet technical, method of Proposition \ref{prop:largekBd}, the argument is shorter and hopefully slightly more illuminating than that of Proposition \ref{prop:largekBd}.) 

In \cite{ManHighOrd}, only the size of $\mc{L}_d(\e)$ was studied, using the ideas of \cite{GSPret} to show that $\mc{L}_d(\e)$ is a $2k$-\emph{sumfree set} (in the sense of additive combinatorics, see e.g. \cite[Thm. 3]{bajnok} and \cite[Thm. 2.4]{ould}). Bounds for $|\mc{L}_d(\e)|$ crucially depended in this way on the divisors of $d$, and ultimately on whether or not $P^-(d)$ was large. Drawing on the ideas used to prove Theorem \ref{thm:impShort}, we obtain significantly more structural information about $\mc{L}_d(\e)$, which enables us to better estimate its size. In this way, we refine \cite[Thm. 3]{ManHighOrd} in a way that does not rely on $P^-(d)$ being large.

\subsection*{Outline of the paper}
The paper is organised as follows. In Section \ref{sec:aux} we collect several results about character sums, and estimates for Ces\`{a}ro and logarithmic mean values of multiplicative functions. We also state and prove a slight refinement of an estimate from \cite{ManHighOrd} establishing a lower bound for the number of primes $p \leq q^{\delta}$ with $\chi(p) \neq 0,1$. \\
In Section \ref{sec:twists} we establish structure theorems for the respective sets $\mc{C}_d(\e)$ and $\mc{L}_d(\e)$ of powers $1 \leq \ell \leq d-1$ for which $|S_{\chi^\ell}(x)| \geq \e x$ and for which $M(\chi^\ell) > \e \sqrt{q}\log q$. In Section \ref{sec:ces}, the structure theorem for $\mc{C}_d(\e)$ is applied to study the Ces\`{a}ro averages of the short sums $S_{\chi^\ell}(x)$. The outcome of the analysis in that section is Theorem \ref{thm:impShort}. In Section \ref{sec:pauc}, we use Theorem \ref{thm:impShort} and several additional ideas to derive Theorem \ref{thm:paucity} and its corollary, Corollary \ref{cor:paucity}. Finally, in Section \ref{sec:max} we use our structure theorem for $\mc{L}_d(\e)$ to establish Theorem \ref{thm:impMax}.

\subsection*{Acknowledgments} We thank Youness Lamzouri and Oleksiy Klurman for useful comments and encouragement.

\section{Auxiliary results} \label{sec:aux}
\subsection{Character sums and mean values of multiplicative functions}
In this section we collect various results about mean values of multiplicative functions in general, and their connection to character sums in particular. Our first lemma shows that if a character $\psi$ has a large maximal sum $M(\psi)$ then $M(\psi)$ is asymptotic to a logarithmically-averaged partial sum determined by $\psi$.  
\begin{prop}[Prop. 2.1 of \cite{GraMan}] \label{prop:MpsiAsymp}
Fix $\Delta \in (2/\pi,1)$ and let $\psi$ be a character modulo $m$. Then
$$
M(\psi) \gg \sqrt{m}(\log m)^{\Delta}
$$
if and only if there is a primitive character $\xi \pmod{\ell}$ with 
$$
\xi(-1) = -\psi(-1) \text{ and } \ell \leq (\log m)^{2(1-\Delta)} (\log\log m)^4
$$ 
such that
$$
\max_{1 \leq N \leq q} \left|\sum_{n \leq N} \frac{(\psi \bar{\xi})(n)}{n}\right| \gg \frac{\phi(\ell)}{\sqrt{\ell}} (\log m)^{\Delta}.
$$
In this case there is a $c_{\psi,\xi} \in [1/2,3]$ such that as $m \ra \infty$,
$$
M(\psi) = (c_{\chi,\xi} + o(1)) \frac{\sqrt{m\ell}}{\pi \phi(\ell)} \max_{1 \leq N \leq q} \left|\frac{(\psi\bar{\xi})(n)}{n}\right|.
$$
\end{prop}
Our next lemma indicates how the estimation of the logarithmically-averaged partial sums of a bounded multiplicative function $f$ are quantitatively controlled by the distribution of the prime values $(f(p))_p$.\\
In the sequel we write $\mb{U} := \{z \in \mb{C} : |z| \leq 1\}$ to denote the closed unit disc in the complex plane. For $y \geq 2$ we define the pretentious distance (at scale $y$) between functions $f,g: \mb{N} \ra \mb{U}$ by
$$
\mb{D}(f,g;y) := \left(\sum_{p \leq y} \frac{1-\text{Re}(f(p)\bar{g}(p))}{p}\right)^{1/2}.
$$
See \cite[Sec.3.1]{ManHighOrd} for a discussion of the properties of the pretentious distance. 
\begin{lem} \label{lem:GSLog}
Let $f: \mb{N} \ra \mb{U}$ be multiplicative and let $x \geq 1$. Then
$$
\max_{1 \leq y \leq x} \left|\sum_{n \leq y} \frac{f(n)}{n}\right| \ll 1 + (\log x) e^{-\frac{1}{2}\mb{D}(f,1;x)^2}.
$$
\end{lem}
\begin{proof}
Let $y_0 \in [1,x]$ maximise the left-hand side. Applying \cite[Lem. 4.3]{GSPret}], we obtain
$$
\max_{1 \leq y \leq x} \left|\sum_{n \leq y} \frac{f(n)}{n}\right| = \left|\sum_{n \leq y_0} \frac{f(n)}{n}\right| \ll 1 + (\log y_0) e^{-\frac{1}{2} \mb{D}(f,1;y_0)^2}.
$$
Since $1 + \text{Re}(f(p)) \geq 0$ for all $p$, by Mertens' theorem we find
$$
(\log y_0) e^{-\frac{1}{2}\mb{D}(f,1;y_0)^2} \asymp \exp\left(\frac{1}{2}\sum_{p \leq y_0} \frac{1+\text{Re}(f(p))}{p}\right)  \leq \exp\left(\frac{1}{2}\sum_{p \leq x} \frac{1+\text{Re}(f(p))}{p}\right) \asymp (\log x) e^{-\frac{1}{2}\mb{D}(f,1;x)^2},
$$
and the claimed bound follows.
\end{proof}
Whereas the logarithmically-averaged partial sum up to $x$ of a bounded multiplicative function $f$ is always controlled by $\mb{D}(f,1;x)$, the same is not true of Ces\`{a}ro-averaged sums in general. In this case, the following estimate will be suitable for our applications to short sums of characters.
\begin{lem}[Hal\'{a}sz-Montgomery-Tenenbaum Inequality]\label{lem:HMT}
Let $f: \mb{N} \ra \mb{U}$ be multiplicative. Let $x\geq 3$, $T \geq 1$, and set
$$
M := \min_{|t| \leq T} \mb{D}(f,n^{it};x)^2.
$$
Then we have
\begin{equation}\label{eq:HMT}
\frac{1}{x}\left|\sum_{n \leq x} f(n)\right| \ll (M+1) e^{-M} + \frac{1}{T} + \frac{\log\log x}{\log x},
\end{equation}
\end{lem}
\begin{proof}
This is \cite[Cor. III.4.12]{Ten}.
\end{proof}
Finally, the following lemma shows that characters of small modulus cannot be too close in pretentious distance to the archimedean characters $n \mapsto n^{it}$.
\begin{lem}\label{lem:repelChar}
There is an absolute constant $C_0 \geq 1$ such that if $C \geq C_0$ the following is true. Let $y \geq 10$ and suppose $1 \leq m \leq y^{1/C}$.  Let $\psi$ be a Dirichlet character modulo $m$ and let $|t| \leq y^2$. If
$$
\mb{D}(\psi,n^{it};y)^2 \leq C
$$
then $\psi$ is principal and $|t|\log y \ll e^{2C}$.
\end{lem}
\begin{proof}
This is \cite[Lem. 3.3]{ManInhom}, made slightly more precise (the bound claimed here follows from the proof of that result). 
\end{proof}

\subsection{Results about high order characters}
We next give a generalisation and sharpening of a result implicitly derived in \cite[Sec. 6]{ManHighOrd} (see in particular the proof of Theorem 2 there), showing that $\chi(p) \neq 0,1$ often when $\chi$ has large order. In the sequel, we write
$$
\Sigma_{\chi}(x) := \sum_{\ss{p \leq x \\ \chi(p) \neq 0,1}} \frac{1}{p}, \quad x \geq 2.
$$
Given a multiplicative function $f: \mb{N} \ra \mb{U}$ and $x \geq 1$ we also put
$$
S_f(x) := \sum_{n \leq x} f(n), \quad L_f(x) := \sum_{n \leq x} \frac{f(n)}{n}.
$$
\begin{prop} \label{prop:LowBdSigma}
Let $\chi$ be a non-principal character modulo $q$ of order $d \geq 2$. Then there is a constant $c_1 > 0$ such that if
$$
\delta := \max\left\{\left(\frac{\log\log(ed)}{c_1\log d}\right)^{1/2}, (\log q)^{-c_1}\right\}
$$
then whenever $x > q^{\delta}$ either
$$
\max\left\{\frac{|S_\chi(x)|}{x}, \frac{|L_{\chi}(q)|}{\log q}\right\} < \frac{1}{(\log d)^{1/5}}, 
$$
or else
$$
\Sigma_{\chi}(x) \geq c_1 \log\log d.
$$
\end{prop}
\begin{proof}
We follow the arguments in \cite[Sec. 5]{ManHighOrd}. Let $c_1 > 0$ be a parameter to be chosen shortly, define $\delta$ as above and let $x > q^{\delta}$. Assume for the sake of contradiction that
\begin{equation}\label{eq:Assump1}
\max\left\{\frac{|S_{\chi}(x)|}{x}, \frac{|L_{\chi}(q)|}{\log q}\right\} \geq \frac{1}{(\log d)^{1/5}},
\end{equation}
and also that
\begin{equation}\label{eq:Assump2}
\Sigma_{\chi}(x) < c_1 \log\log d.
\end{equation}
Set $c := 4c_1$ and consider first the case in which $d \leq e^{(\log q)^c}$.
Since $\chi$ has order $d$, orthogonality of Dirichlet characters implies that
\begin{equation}\label{eq:equidist}
|\{n \leq q : \chi(n) = 1\}| = \frac{1}{d}\sum_{0 \leq j \leq d-1} \sum_{\ss{n \leq q \\ (n,q) = 1}} \chi(n)^j = \frac{\phi(q)}{d}.
\end{equation}
Now, let $g: \mb{N} \ra [0,1]$ be the completely multiplicative function defined at primes by
$$
g(p) := \begin{cases} 1 &\text{ if } p \leq x \text{ and } \chi(p) = 1 \\
0 & \text{ otherwise.} \end{cases}
$$
Write also
$$
u := \mb{D}(g,1;q)^2 = \sum_{x < p \leq q} \frac{1}{p} + \sum_{\ss{p \leq x \\ \chi(p) \neq 1}} \frac{1}{p},
$$
set $\sg_-(u) := u\rho(u)$ where $\rho$ is the Dickman function (see e.g. \cite[Sec. III.5.3-4]{Ten} for a definition and relevant properties), and put
$$
R(g;q) := \prod_{p \leq q} \left(1-\frac{1}{p}\right)\left(1-\frac{g(p)}{p}\right)^{-1} \asymp e^{-u}. 
$$
By Hildebrand's theorem \cite[Thm. 2]{Hil} there are absolute constants $\beta \in (0,1)$ and $A > 0$ such that 
\begin{align} \label{eq:Hild}
|\{n \leq q: \chi(n) = 1\}| &\geq \sum_{n \leq q} g(n) \geq
A q R(g;q) \left(\sg_-(e^u) + O(e^{-(\log q)^\beta})\right).
\end{align}
Since $\rho(v) \geq v^{-2v}$ for large $v$ we have $\sg_-(e^u) \gg e^{-(\log q)^{\beta}}$ as long as
$u \leq \tfrac{\beta}{2} \log\log q$. Since $x > q^\delta$ by assumption,
\begin{align*}
u \leq \sum_{\ss{p \leq x \\ \chi(p) \neq 1}} \frac{1}{p} + \log(1/\delta) + O(1) \leq \Sigma_{\chi}(x) + (c_1 + o(1))\log\log q \leq (2c_1 + o(1)) \log\log q
\end{align*}
we obtain $\sg_-(e^u) \gg e^{-(\log q)^{\beta}}$ with a suitably large implicit constant as long as $c = 4c_1 < \beta$, and $q$ is large enough. Combining \eqref{eq:equidist} and \eqref{eq:Hild}, we find
$$
\frac{\phi(q)}{d} \gg q e^{-u} \sg_-(e^u) \gg q \rho(e^u) \gg q e^{-2ue^u}.
$$
We deduce that
$$
u + \log(2u) = (1+o(1))u \geq \log\log d,
$$
whence also
$$
\Sigma_{\chi}(x) + \sum_{p|q} \frac{1}{p} + \log(1/\delta) \geq (1-o(1))\log\log d.
$$
As $\log(1/\delta) \leq \frac{1}{2} \log\log d$ we deduce that
\begin{equation}\label{eq:Altern}
\max\{\Sigma_{\chi}(x) , \log(q/\phi(q))\} \geq (1/4-o(1)) \log\log d.
\end{equation}
Now using a theorem of Hall \cite{Hall}, we have 
$$
|S_{\chi}(x)| \leq \sum_{n \leq x} 1_{(n,q) = 1} \ll x\prod_{\ss{p \leq x \\ p | q}} \left(1-\frac{1}{p}\right) \ll x\frac{\phi(q)}{q}\exp\left(\sum_{\ss{p > x \\ p | q}} \frac{1}{p}\right) \ll x\frac{\phi(q)}{q}.
$$
Similarly, 
$$
|L_{\chi}(q)| \leq \sum_{\ss{p|n \Rightarrow p \leq q \\ (n,q)=1}} \frac{1}{n} \ll (\log q)\prod_{p |q} \left(1-\frac{1}{p}\right) = \frac{\phi(q)}{q}\log q.
$$
From \eqref{eq:Assump1} we see that
$$
\sum_{p|q} \frac{1}{p} = \log(q/\phi(q)) + O(1) \leq \frac{1}{5}\log\log d + O(1).
$$
In light of \eqref{eq:Altern}, we have
$$
\Sigma_{\chi}(x) \geq (1/4-o(1)) \log\log d,
$$
whenever $d \leq e^{(\log q)^{c}}$, another contradiction as $c_1 < \beta/4 < 1/4$. Hence $\Sigma_{\chi}(x) > c_1 \log\log d$ in this case. \\
Next, assume that $d \geq e^{(\log q)^{c}}$. The argument\footnote{While the result stated there was only stated for prime $q$, the proof employed zero-density estimates that hold for the family of non-principal characters to more general moduli $q$, and is therefore applicable to the present circumstances.} in the proof of \cite[Prop. 5.1]{ManHighOrd} actually shows in this case that
$$
\sum_{\ss{p \leq q \\ \chi(p) \neq 1}} \frac{1}{p} \geq (c-o(1))\log\log q.
$$
Since we have
$$
\sum_{p|q} \frac{1}{p} = \log(q/\phi(q)) + O(1) \leq \log\log\log q + O(1)
$$
and $c_1 = c/4$ we deduce that when $d \geq e^{(\log q)^c}$,
$$
\Sigma_{\chi}(x) \geq \sum_{\ss{p \leq q \\ \chi(p) \neq 0,1}} \frac{1}{p} - \sum_{x < p \leq q} \frac{1}{p} \geq (c-o(1)) \log\log q - \log(1/\delta) \geq (3c_1-o(1)) \log\log q > c_1 \log\log d
$$
using $q \geq \phi(q) \geq d$ in the last bound. The claim now follows. 


\end{proof}
\section{Minimising Archimedean and Dirichlet twists} \label{sec:twists}

Let $q \geq 3$ be large and let $\chi$ be a primitive character modulo $q$ with order $d\geq 2$. We assume that $d$ is larger than any fixed absolute constant, and for $c_1 > 0$ chosen as in Proposition \ref{prop:LowBdSigma} we let
$$
\delta := \max\left\{\left(\frac{\log\log(ed)}{c_1\log d}\right)^{1/2}, (\log q)^{-c_1}\right\}.
$$
Let $\e > 0$ and $x > q^{\delta}$, and define the sets\footnote{As we define $M(\psi) := \max_{1 \leq t \leq m} |S_{\psi}(t)|$ for a character $\psi$ of modulus $m$, the maximal sum $M(\chi_0)$ of the principal character $\chi_0 \pmod{q}$ is well-defined and equal to $q-1$.}
$$
\mc{C}_d(\e) := \{\ell \pmod{d} : |S_{\chi^\ell}(x)| \geq \e x\}, \quad \mc{L}_d(\e) := \{\ell \pmod{d} : M(\chi^\ell) \geq \e \sqrt{q}\log q \}.
$$
We assume here that $\e$ is smaller than any fixed constant. We will prove a structural result about each of the sets $\mc{C}_d(\e)$ and $\mc{L}_d(\e)$. The first is an analogue of \cite[Prop. 4.1]{ManHighOrd}.
\begin{prop} \label{prop:shortStruc}
Let $\e > 0$ with $\e \geq (\log q)^{-1/10}$. Assume that $|\mc{C}_d(\e)| \geq \e d$. Then there are positive integers $1 \leq g \leq \e^{-1}$ and
and $1 \leq m \leq \e^{-2}$ such that 
$$
\mc{C}_d(\e) \subseteq \{\ell g \pmod{d} : 1 \leq \ell \leq d/g\} = m\mc{C}_d(\e),
$$ 
and furthermore
$$
\max_{1 \leq \ell \leq d/g} \mb{D}(\chi^{g\ell},1;x)^2 \leq 200m^2 \log(1/\e).
$$
\end{prop}
\begin{proof}
Observe that $\ell \in \mc{C}_d(\e)$ if and only if $-\ell \in \mc{C}_d(\e)$, so $\mc{C}_d(\e)$ is a symmetric subset of $\mb{Z}/d\mb{Z}$. Applying \cite[Lem. 5.8]{ManInhom}, we find an integer\footnote{The proof of \cite[Lem. 5.8]{ManInhom} actually shows that $m = 2^{j+1}$, where $j$ is the largest integer such that $(3/2)^{j-1} \leq 1/\e$. Since $2\log(3/2) \geq \log 2$, it is easy to check that this forces $m \leq \e^{-2}$ when $\e$ is sufficiently small.} $1 \leq m \leq \e^{-2}$ and a divisor $r|d$ with $r \geq \e d$ such that $m\mc{C}_d(\e) = H$ is a subgroup of $\mb{Z}/d\mb{Z}$ of order $|H| = r$. Note that $H$ is generated by $d/r =: g$, so that $1 \leq g \leq \e^{-1}$ and $H$ can be parameterised as $\{\ell g \pmod{d} : 1 \leq \ell \leq r\}$. Moreover, since $0 \in \mc{C}_d(\e)$ we have $\mc{C}_d(\e) \subseteq m \mc{C}_d(\e)$ for all $m \geq 1$, and hence $\mc{C}_d(\e) \subseteq H$ as required. \\
The proof of Proposition \ref{prop:shortStruc} now follows the same lines as that of \cite[Prop. 4.1]{ManHighOrd}. Taking $T = 1/\e^2$ and applying Lemma \ref{lem:HMT}, we find that for each $\ell \in \mc{C}_d(\e)$,
$$
\e x \leq |S_{\chi^\ell}(x)| \ll x \left(\mb{D}(\chi^\ell,n^{i\tilde{t}_\ell};x)^2 e^{-\mb{D}(\chi^\ell,n^{i\tilde{t}_{\ell}};x)^2} + \e^2 + \frac{\log\log x}{\log x}\right)
$$
for some $|\tilde{t}_{\ell}| \leq 1/\e^2$, from which we deduce that (when $d$, and thus $q$, is sufficiently large)
$$
\max_{\ell \in \mc{C}_d(\e)} \mb{D}(\chi^\ell,n^{i\tilde{t}_\ell};x) \leq \sqrt{2\log(1/\e)}.
$$
Now for $\ell_1,\ell_2 \in H$ we can choose representations
$$
\ell_1 \equiv r_1 + \cdots + r_m \pmod{d}, \quad \ell_2 \equiv s_1 + \cdots + s_m \pmod{d}, \quad r_i,s_j \in \mc{C}_d(\e).
$$
Setting
$$
t(\ell_1) := \tilde{t}_{r_1} + \cdots + \tilde{t}_{r_m}, \quad t(\ell_2) := \tilde{t}_{s_1} + \cdots + \tilde{t}_{s_m},
$$
we find by the pretentious triangle inequality that
\begin{align*}
\mb{D}(\chi^{\ell_1},n^{it(\ell_1)}; x) &\leq \sum_{1 \leq j \leq m} \mb{D}(\chi^{r_j},n^{i\tilde{t}_{r_j}};x) \leq \sqrt{2m^2\log(1/\e)}, \\
\mb{D}(\chi^{\ell_2},n^{it(\ell_2)}; x) &\leq \sum_{1 \leq j \leq m} \mb{D}(\chi^{s_j},n^{i\tilde{t}_{s_j}};x) \leq \sqrt{2m^2\log(1/\e)}, \\
\mb{D}(\chi^{\ell_1+\ell_2},n^{i(t(\ell_1)+t(\ell_2))}; x) &\leq \mb{D}(\chi^{\ell_1}, n^{it(\ell_1)};x) + \mb{D}(\chi^{\ell_2}, n^{it(\ell_2)};x) 
\leq \sqrt{8m^2\log(1/\e)}.
\end{align*}
Define the map $\phi: H \ra \mb{R}$ via $\phi(\ell) := t_{\ell}$, where for each $\ell \in H$, $t_{\ell} \in [-2m/\e^2,2m/\e^2]$ is chosen such that
$$
\mb{D}(\chi^\ell, n^{it_{\ell}};x) = \min_{|t| \leq 2m/\e^2} \mb{D}(\chi^\ell, n^{it};x).
$$
Since $|t(\ell_1)+t(\ell_2)| \leq |t(\ell_1)| + |t(\ell_2)| \leq 2m/\e^2$ we see by the minimality property of $\phi(\ell_j)$, $j = 1,2$, that
\begin{align*}
\mb{D}(n^{i\phi(\ell_j)}, n^{it(\ell_j)};x) &\leq \mb{D}(\chi^{\ell_j},n^{it(\ell_j)};x) + \mb{D}(\chi^{\ell_j},n^{i\phi(\ell_j)};x) \leq 2 \mb{D}(\chi^{\ell_j},n^{it(\ell_j)};x) \leq \sqrt{8m^2 \log(1/\e)}, \end{align*}
and also that
\begin{align*}
\mb{D}(n^{i\phi(\ell_1+\ell_2)}, n^{i(t(\ell_1) + t(\ell_2))};x) &\leq \mb{D}(\chi^{\ell_1+\ell_2}, n^{i\phi(\ell_1+\ell_2)};x) + \mb{D}(\chi^{\ell_1+\ell_2}, n^{i(t(\ell_1)+t(\ell_2))};x) \\
&\leq 2\mb{D}(\chi^{\ell_1+\ell_2},n^{i(t(\ell_1)+t(\ell_2)}; x) \leq \sqrt{32 m^2 \log(1/\e)}.
\end{align*}
Set $K := \e^{-16m^2} \geq 1$. It follows from Lemma \ref{lem:repelChar} that
$$
|\phi(\ell_j) - t(\ell_j)| \leq \frac{K}{\log x}, \quad |\phi(\ell_1+\ell_2) - t(\ell_1) -t(\ell_2)| \leq \frac{K^4}{\log x},
$$
whence that also
$$
|\phi(\ell_1+\ell_2) - \phi(\ell_1)-\phi(\ell_2)| \leq \frac{2K + K^4}{\log x} \leq \frac{3K^4}{\log x}.
$$
The map $\phi$ is therefore a $\tfrac{3K^4}{\log x}$-approximate homomorphism, in the sense of \cite{Ruz}. Since there are no non-zero homomorphisms $H \to \mb{R}$, by \cite[Statement (7.2)]{Ruz} we deduce that 
$$
\max_{\ell \in H}|\phi(\ell)| \leq \frac{3K^4}{\log x}.
$$
By Mertens' theorem we then deduce (again using the minimality property of $\phi(\ell)$) that
\begin{align*}
\max_{\ell \in H} \mb{D}(\chi^\ell,1;x)^2 &\leq 2 \max_{\ell \in H} \left(\mb{D}(\chi^{\ell},n^{i\phi(\ell)};x)^2 + \mb{D}(1,n^{i\phi(\ell)};x)^2\right) \\
&\leq 2\max_{\ell \in H} \left(\mb{D}(\chi^\ell, n^{it(\ell)};x)^2 + \log(1+|\phi(\ell)| \log x) + O(1)\right) \\
&\leq 2\left(8m^2 \log(1/\e) + 4 \log K + O(1)\right) \\
&< 150 m^2 \log(1/\e) + O(1),
\end{align*}
and the claim follows since each $\ell \in H$ can be written in the form $g \ell'$ with $1 \leq \ell' \leq r$.
\end{proof}
We also prove an analogous pretentiousness result about $\mc{L}_d(\e)$; no such result appeared in \cite{ManHighOrd}.
\begin{prop} \label{prop:maxStruc}
Let $\e > 0$ with $\e \geq (\log q)^{-1/10}$. Assume that $|\mc{L}_d(\e)| \geq \e d$. Then there are 
\begin{itemize}
\item positive integers $1 \leq m\leq \e^{-2}$ and $1 \leq g \leq \e^{-1}$ such that
$$
\mc{L}_d(\e) \subseteq \{g \ell : 1 \leq \ell \leq d/g\} = m \mc{L}_d(\e),
$$
\item a positive integer $1 \leq k \leq \e^{-3m}$ and a primitive Dirichlet character $\xi \pmod{k}$ of order dividing $d/g$ 
such that
$$
\max_{1 \leq \ell \leq d/g} \mb{D}(\chi^{g\ell},\xi^\ell;x)^2 \ll m^2 \log(1/\e).
$$
\end{itemize}
\end{prop}
\begin{proof}
As with $\mc{C}_{d}(\e)$, $\mc{L}_d(\e)$ is a symmetric subset of $\Z/d\Z$. Applying \cite[Lem. 5.8]{ManInhom} as in the proof of the previous proposition, we can find an integer $1 \leq m \leq \e^{-2}$ and a divisor $r | d $ with $r \geq \e d$ such that $m\mc{L}_d(\e) = H$ is a subgroup of $\Z/d\Z$ of order $|H| = r$, and $\mc{L}_d(\e) \subseteq H$. By  Proposition \ref{prop:MpsiAsymp}, 
for each $\ell \in \mc{L}_d(\e)$
there is a primitive character $\xi_{\ell}$ (mod $k_{\ell}$) such that  
\begin{align} \label{eq:LdepsBd}
\e \log q \leq \frac{1}{\sqrt{q}}M(\chi^{\ell}) \ll \frac{\sqrt{k_\ell}}{\phi(k_\ell)} \max\limits_{1\leq N \leq q}|L_{\chi^\ell\bar{\xi_\ell}}(N)|.
\end{align}
Applying the trivial bound $|L_{\chi^\ell \bar{\xi}_\ell}(N)| \leq \log N$ in \eqref{eq:LdepsBd}, we see that
$$
\e \ll \frac{\sqrt{k_\ell}}{\phi(k_\ell)}\max_{1 \leq N \leq q} \frac{\log N}{\log q} \ll \frac{\sqrt{k_{\ell}}}{\phi(k_\ell)},
$$
so that since $\phi(b) \gg b/\log\log (3b)$ for any $b \geq 1$ we obtain the bound
\begin{equation}\label{eq:modBd}
\max_{\ell \in \mc{L}_d(\e)} k_{\ell} \ll (\e^{-1} \log\log(1/\e))^2.
\end{equation}
Using instead Lemma \ref{lem:GSLog} in  \eqref{eq:LdepsBd}, we deduce that
$$
\e \log q \ll 1 + (\log q)e^{-\tfrac{1}{2}\mb{D}(\chi^\ell\bar{\xi}_\ell,1;q)^2},
$$
from which we find that
\begin{align*}
\mb{D}(\chi^\ell\bar{\xi_\ell}, 1;q)^2 = \mb{D}(\chi^\ell, \xi_\ell;q)^2 \ll \log (1/\e).
\end{align*}
Let $\ell' \in H$. As $H = m\mc{C}_d(\e)$ there is a representation
\begin{align*}
\ell' \equiv r_1 + \cdots + r_m \pmod{d}, \;\; r_i \in \mc{L}_d(\e).
\end{align*}
By the pretentious triangle inequality, we find that
\begin{align*}
\mb{D}(\chi^{\ell'}, \prod\limits_{i = 1}^m\xi_{r_i}; q) \leq \sum\limits_{1 \leq i \leq m}\mb{D}(\chi^{r_i}, \xi_{r_i}; q) \ll m\sqrt{\log (1/\e)}.
\end{align*}
Using \eqref{eq:modBd}, the modulus of $\prod_{1 \leq i \leq m} \xi_{r_i}$ is
$$
[k_{r_1},\ldots,k_{r_m}] \leq k_{r_1} \cdots k_{r_m} \leq (\e^{-1}\log\log(1/\e))^{2m} \leq \e^{-3m},
$$
provided $\e$ is sufficiently small. This modulus is a divisor of 
$$
K(\e) := [1,2,\ldots,\lfloor (\e^{-1}\log\log(1/\e))^2\rfloor],
$$
which, by the prime number theorem, satisfies
$$
\log K(\e) \leq 2(\e^{-1}\log\log(1/\e))^2,
$$
again provided $\e$ is sufficiently small. \\
In the sequel, write $\psi_0$ to denote the principal character modulo $K(\e)$. For each $\ell' \in H$, there is a primitive character $\xi_{\ell'}$ with conductor $\leq \e^{-3m}$ such that
\begin{align*}
\mb{D}(\chi^{\ell'}, \xi_{\ell'}; q)^2 \ll m^2 \log(1/\e) + \sum_{\ss{p | k_{\ell'}}} \frac{1}{p} \ll m^2 \log(1/\e) + \log\log\log k_{\ell'} \ll m^2 \log(1/\e),
\end{align*}
and therefore also
\begin{equation}\label{eq:withPsi0}
\mb{D}(\chi^{\ell'},\xi_{\ell'}\psi_0; q)^2 \leq \mb{D}(\chi^{\ell'},\xi_{\ell'};q)^2 + \sum_{p|K(\e)} \frac{1}{p}  \ll m^2 \log(1/\e) + \log\log\log K(\e) \ll m^2 \log(1/\e).
\end{equation}
Let $g = d/r$. Then $H$ can be parametrised as $\{jg \pmod{d} : 1 \leq j \leq r\}$. Let $\mc{S}$ denote the group of characters modulo $[k_{\ell} : \ell \in H]$, generated by the set
\begin{align*}
\{\xi_1 \psi_0,\ldots,\xi_r\psi_0\}.
\end{align*}
Define a map $\phi: H \to \mc{S}$ by $\phi(i) = \xi_i\psi_0$, and $1 \leq i,j \leq r$. By \eqref{eq:withPsi0} and the triangle inequality,
\begin{align*}
\mb{D}(\chi^{(i+j)g}, \phi(i)\phi(j); q) \leq \mb{D}(\chi^{ig}, \phi(i); q) + \mb{D}(\chi^{jg}, \phi(j); q) \ll m\sqrt{\log (1/\e)},
\end{align*}
so that
\begin{align*}
\mb{D}(\phi(i+j), \phi(i)\phi(j); q) \leq \mb{D}(\chi^{(i+j)g}, \phi(i)\phi(j); q) + \mb{D}(\chi^{(i+j)g}, \phi(i+j); q) \ll m\sqrt{\log (1/\e)}.
\end{align*}
Take $k := \max\{k_i,k_j,k_{i+j}\}$. By Lemma \ref{lem:repelChar} we see that either
$$
\e^{-9m}K(\e) \geq k^3 K(\e) > q^{1/C},
$$
or else that 
$$
\phi(i+j) \overline{\phi(i)} \overline{\phi(j)} = \xi_{i+j} \bar{\xi}_i \bar{\xi}_j \psi_0
$$ 
is principal, with modulus $m_{i,j}$ dividing $[K(\e),k_i,k_j,k_{i+j}] = K(\e)$. Since $m^2 \log(1/\e) = o(\log\log q)$ the former is not possible. Thus, writing $\chi_0^{(m_{i,j})}$ to denote the principal character modulo $m_{i,j}$, we have
$$
\phi(i+j) = \xi_{i+j}\psi_0 = \xi_i \xi_j \chi_0^{(m_{i,j})} \psi_0 = (\xi_i \psi_0) (\xi_j \psi_0) = \phi(i)\phi(j).
$$
%
It follows that $\phi$ is a homomorphism, and therefore
$$
\xi_j \psi_0 = \phi(j) = \phi(1)^j = \xi_1^j \psi_0 \text{ for all } 1 \leq j \leq r.
$$
Since $\xi_j$ is primitive for all $1 \leq j \leq r$, it follows that $\xi_r$ is the trivial character, and thus $\xi_1^r$ is principal, i.e., $\xi_1$ has order dividing $r$. \\
We deduce that, uniformly over $1 \leq j \leq r$,
\begin{align*}
\mb{D}(\chi^{jg}, (\xi_1)^j; q)^2 \leq \mb{D}(\chi^{jg},\phi(j);q)^2 + \sum_{p|K(\e)} \frac{1}{p} \ll m^2\log (1/\e).
\end{align*}
We write $\xi = \xi_1$ and $k = k_1$, and the claim follows.
\end{proof}

\section{Averaged Ces\`{a}ro sums: Proof of Theorem \ref{thm:impShort}} \label{sec:ces}
Write $\zeta = e(1/d)$, which is a primitive $d$th root of unity. As in \cite{ManHighOrd} we decompose
$$
\Sigma_{\chi}(x) = \sum_{\ss{p \leq x \\ \chi(p) \neq 0,1}} \frac{1}{p} = \sum_{1 \leq j \leq d-1} \sg_j(x),
$$
where given $y \geq 2$ and $1 \leq j \leq d-1$ we set 
$$
\sg_j(y) := \sum_{\ss{p \leq y \\ \chi(p) = \zeta^j}} \frac{1}{p}.
$$
We consider below how the size of the fibred sums $\sg_j$ influence the magnitudes of the partial sums $S_{\chi^\ell}(x)$, for $1 \leq \ell \leq d-1$. For the remainder of this section, fix $\eta \in (0,1)$ to be a small parameter, to be chosen later. 
\subsection{Small $\sg_j$ case} 
In the case where all $\sg_j$ are ``small'' (in a sense to be made precise), we will prove the following analogue of \cite[Lem. 4.4]{ManHighOrd}.
\begin{prop} \label{prop:smallSig}
Let $\eta$ be sufficiently small, and suppose $\sg_j(x) \leq \tfrac{\eta}{g}\Sigma_{\chi}(x)$ for all $1 \leq j \leq d-1$. Then there are 
elements $1 \leq \ell \leq d/g$ such that
$$
\mb{D}(\chi^{g\ell},1;x)^2 \geq \frac{1}{2} \Sigma_{\chi}(x).
$$
\end{prop}
\begin{proof}
Set $r := d/g$ and let $1 \leq \ell \leq r$. As in \cite{ManHighOrd}, we use the lower bound
\begin{equation}\label{eq:toFE}
\mb{D}(\chi^{g\ell},1;x)^2 = \sum_{1 \leq j \leq d-1} (1-\cos(2\pi j g \ell/d)) \sum_{\ss{p \leq x \\ \chi(p) = \zeta^j}} \frac{1}{p} \geq 8 \sum_{1 \leq j \leq d-1} \left\|\frac{j\ell}{r}\right\|^2 \sg_j(x).
\end{equation}
In the sequel we write $\sg_j = \sg_j(x)$ for convenience. We then consider the (restricted) variance\footnote{In contrast to \cite{ManHighOrd}, here we restrict ourselves to an average over powers $g\ell$, and thus our sums over $\ell$ and $j$ have different ranges.}
$$
\Delta := \frac{1}{r}\sum_{1 \leq \ell \leq r} \left(\sum_{1 \leq j \leq d-1} \left\|\frac{j\ell}{r}\right\|^2 \sg_j - \frac{1}{12} \Sigma_{\chi}(x)\right)^2 = \frac{1}{r}\sum_{1 \leq \ell \leq r} \left(\sum_{1 \leq j \leq d-1} \left(\left\|\frac{j\ell}{r}\right\|^2 - \frac{1}{12}\right)\sg_j\right)^2
$$
Expanding the square and using the Fourier expansion
$$
\|t\|^2 = \frac{1}{12} + \frac{1}{2\pi^2}\sum_{v \neq 0} \frac{(-1)^v}{v^2}e(vt),
$$
for the $1$-periodic map $t \mapsto \|t\|^2$, we obtain 
\begin{align*}
\Delta &= \sum_{1 \leq j_1,j_2 \leq d-1} \sg_{j_1}\sg_{j_2} \frac{1}{r} \sum_{1 \leq \ell \leq r} \left(\left\|\frac{\ell j_1}{r}\right\|^2 - \frac{1}{12}\right)\left(\left\|\frac{\ell j_2}{r}\right\|^2 - \frac{1}{12}\right)  \\
&= \frac{1}{4\pi^4} \sum_{1 \leq j_1,j_2 \leq d-1} \sg_{j_1}\sg_{j_2} \sum_{v_1,v_2 \neq 0} \frac{(-1)^{v_1 + v_2}}{(v_1v_2)^2} \frac{1}{r} \sum_{1 \leq \ell \leq r}e\left(\frac{\ell}{r}(j_1v_1-j_2v_2)\right) \\
&= \frac{1}{4\pi^4} \sum_{1 \leq j_1,j_2 \leq d-1} \sg_{j_1}\sg_{j_2} \sum_{\ss{v_1,v_2 \neq 0 \\ j_1v_1\equiv j_2v_2 \pmod{r}}} \frac{(-1)^{v_1+v_2}}{(v_1v_2)^2} \\
&= \frac{1}{4\pi^4}\sum_{e_1,e_2 | r} \sum_{\ss{1 \leq j_1 < d \\ (j_1,r) = e_1}}\sum_{\ss{1 \leq j_2 < d \\ (j_2,r) = e_2}}\sg_{j_1}\sg_{j_2} \sum_{\ss{v_1,v_2 \neq 0 \\ j_1 v_1\equiv j_2 v_2 \pmod{r}}} \frac{(-1)^{v_1+v_2}}{(v_1v_2)^2},
\end{align*}
where we have split the sum according to the GCDs of $j_1$ and $j_2$ with $r$. Note that $(j_1v_1,r) = (j_2v_2,r)$ whenever $j_1v_1 \equiv j_2v_2 \pmod{r}$. Letting $\lambda = (j_i v_i,r)$ denote this common divisor of $r$, we of course have $e_1,e_2|\lambda$. In this case, writing $\lambda = e_if_i$ and $j_i = J_ie_i$ for $i = 1,2$, and noting that $(J_i,r/e_i) = 1$, we find
$$
\lambda = (e_i \cdot J_i v_i, e_i \cdot r/e_i) = e_i(J_iv_i,r/e_i) = e_i (v_i,r/e_i),
$$
so that $f_i = (v_i, r/e_i)$; in particular, $f_i|v_i$ for $i = 1,2$. Setting $u_i := v_i/f_i$, it follows that
\begin{align*}
\Delta &= \frac{1}{4\pi^4} \sum_{\ss{\lambda | r }} \sum_{\ss{e_if_i = \lambda \\ i = 1,2}} \frac{1}{(f_1f_2)^2} \sum_{\ss{1 \leq J_1 < d/e_1 \\ (J_1,r/e_1) = 1}}\sum_{\ss{1 \leq J_2 < d/e_2 \\ (J_2,r/e_2) = 1}} \sg_{J_1e_1}\sg_{J_2e_2} \sum_{\ss{u_1,u_2 \neq 0 \\ J_1u_1\equiv J_2u_2 \pmod{r/\lambda} \\ (u_1,r/\lambda) = (u_2,r/\lambda) = 1}} \frac{(-1)^{f_1u_1+f_2u_2}}{(u_1u_2)^2}.
\end{align*}
If we truncate both of the $u_1,u_2$ series to terms with $|u_j| \leq Mr/\lambda$ for some $M \geq 1$ then we obtain
\begin{equation}\label{eq:deltaHEM}
\Delta = \frac{1}{4\pi^4} \sum_{\ss{\lambda | r }} \sum_{\ss{e_if_i = \lambda \\ i = 1,2}} \frac{1}{(f_1f_2)^2} \sum_{\ss{1 \leq J_1 < d/e_1 \\ (J_1,r/e_1) = 1}}\sum_{\ss{1 \leq J_2 < d/e_2 \\ (J_2,r/e_2) = 1}} \sg_{J_1e_1}\sg_{J_2e_2} \sum_{\ss{1 \leq |u_1|,|u_2| \leq Mr/\lambda \\ J_1u_1 \equiv J_2u_2 \pmod{r/\lambda} \\ (u_1,r/\lambda) = (u_2,r/\lambda) =1}} \frac{(-1)^{f_1u_1+f_2u_2}}{(u_1u_2)^2} + O\left(\mc{E}_M\right),
\end{equation}
where we have set
$$
\mc{E}_M = \frac{1}{M r} \sum_{\ss{\lambda | r}} \lambda \sum_{\ss{e_if_i = \lambda \\ i = 1,2}} \frac{1}{(f_1f_2)^2} \sum_{\ss{1 \leq J_1 < d/e_1 \\ (J_1,r/e_1) = 1}}\sum_{\ss{1 \leq J_2 < d/e_2 \\ (J_2,r/e_2) = 1}} \sg_{J_1e_1}\sg_{J_2e_2}.
$$
Rearranging, we obtain
\begin{align*}
\mc{E}_M &= \frac{1}{Mr} \sum_{\ss{e_1,e_2|r}} \left(\sum_{\ss{1 \leq j_1 < d \\ (j_1,r) = e_1}} \sg_{j_1}\right)\left(\sum_{\ss{1 \leq j_2 < d \\ (j_2,r) = e_2}} \sg_{j_2}\right) \sum_{\ss{\lambda | r \\ [e_1,e_2] | \lambda}} \frac{\lambda}{(\lambda^2/e_1e_2)^2} \\
&= \frac{1}{Mr} \sum_{\ss{e_1,e_2|r}} \left(\sum_{\ss{1 \leq j_1 < d \\ (j_1,r) = e_1}} \sg_{j_1}\right)\left(\sum_{\ss{1 \leq j_2 < d \\ (j_2,r) = e_2}} \sg_{j_2}\right) \sum_{\ss{a | r/[e_1,e_2]}} \frac{(e_1e_2)^2}{a^3[e_1,e_2]^3}.
\end{align*}
We observe that for each pair $e_1,e_2|r$ we have
$$
\frac{(e_1e_2)^2}{r[e_1,e_2]^3} \sum_{a|r/[e_1,e_2]} \frac{1}{a^3} \ll \frac{(e_1,e_2)^3}{re_1e_2} \leq 1,
$$
since $(e_1,e_2) \leq \min\{e_1,e_2\} \leq r$. It follows therefore that
$$
\mc{E}_M \ll \frac{1}{M} \sum_{\ss{e_1,e_2|r}} \left(\sum_{\ss{1 \leq j_1 < d \\ (j_1,r) = e_1}} \sg_{j_1}\right)\left(\sum_{\ss{1 \leq j_2 < d \\ (j_2,r) = e_2}} \sg_{j_2}\right) = \frac{1}{M} \Sigma_{\chi}(x)^2.
$$
In the main term in \eqref{eq:deltaHEM}, momentarily fix $1 \leq |u_1|,|u_2| \leq Mr/\lambda$ with $(u_1,r/\lambda) = (u_2,r/\lambda) = 1$, assume $e_2 \geq e_1$ and let $1 \leq J_1 < d/e_1$ with $(J_1,r/e_1) = 1$. If  
$$
J_2 u_2 \equiv J_1u_1 \pmod{r/\lambda}, \text{ i.e., } J_2 \equiv J_1 u_2^{-1} u_1 \pmod{r/\lambda}
$$
then among the $d/e_2$ possible choices of $J_2$ there are $O(\frac{d/e_2}{r/\lambda} + 1)$ choices satisfying this congruence. For each of these $J_2$ we have $\sg_{J_2e_2} \leq (r\eta/d) \Sigma_{\chi}(x)$ by assumption. Thus, applying the preceding arguments to all $u_1,u_2$, $e_2 \geq e_1$ (the other case being identical up to relabelling) and $J_1$ as above the main term in \eqref{eq:deltaHEM} is bounded above by 
\begin{align} \label{eq:lowGCDpart}
&\ll \frac{r\eta}{d}  \Sigma_{\chi}(x) 
\sum_{\ss{\lambda | r }} \sum_{\ss{e_1f_1 = \lambda \\ e_2f_2 = \lambda \\ e_2 \geq e_1}} \frac{1}{(f_1f_2)^2} \sum_{\ss{1 \leq j_1 < d \\ (j_1,r) = e_1}} \sg_{j_1} \sum_{\ss{1 \leq |u_1|, |u_2| \leq Mr/\lambda  \\ (u_1u_2, r/\lambda) = 1}} \frac{1}{(u_1u_2)^2} \cdot \left(\frac{\lambda d}{r e_2} + 1\right) \nonumber \\
&\ll \frac{r\eta}{d}  \Sigma_{\chi}(x)(\mc{R}_1 + \mc{R}_2),
\end{align}
where, using $\lambda = e_1f_1 = e_2f_2$ and the fact that the series in $u_1,u_2$ are both convergent, we have set
\begin{align*}
\mc{R}_1 &:= \frac{d}{r}\sum_{e_1|r} \left(\sum_{\ss{1 \leq j_1 < d \\ (j_1,r) = e_1}} \sg_{j_1}\right) 
\sum_{f_1|r/e_1} \frac{1}{f_1^2} \sum_{\ss{f_2|e_1f_1 \\ f_2 \leq f_1}} \frac{1}{f_2}, \\
\mc{R}_2 &:= \sum_{e_1|r} \left(\sum_{\ss{1 \leq j_1 < d \\ (j_1,r) = e_1}} \sg_{j_1}\right) \sum_{\ss{e_2|r \\ e_2 \geq e_1}} \sum_{\ss{\lambda|r \\ [e_1,e_2]|\lambda}} \frac{(e_1e_2)^2}{\lambda^4}.
\end{align*}
To estimate $\mc{R}_1$ we observe that $1_{f_2 \leq f_1} \leq (f_1/f_2)^{1/2}$ in the inner sum, whence
$$
\sum_{f_1|r/e_1} \frac{1}{f_1^2} \sum_{\ss{f_2|e_1f_1 \\ f_2 \leq f_1}} \frac{1}{f_2} \leq \sum_{f_1|r/e_1} \frac{1}{f_1^{3/2}} \sum_{\ss{f_2|e_1f_1}} \frac{1}{f_2^{3/2}} \ll 1.
$$
It follows therefore that
$$
\mc{R}_1 \ll \frac{d}{r}\sum_{e_1|r} \left(\sum_{\ss{1 \leq j_1 < d \\ (j_1,r) = e_1}} \sg_{j_1}\right) = \frac{d}{r} \Sigma_{\chi}(x).
$$
To bound $\mc{R}_2$, note that for each $e_1|r$,
$$
\sum_{\ss{e_2|r \\ e_2 \geq e_1}} \sum_{\ss{\lambda|r \\ [e_1,e_2]|\lambda}} \frac{(e_1e_2)^2}{\lambda^4} 
= \sum_{\ss{e_2|r \\ e_2 \geq e_1}} \frac{(e_1e_2)^2}{[e_1,e_2]^4}\sum_{\ss{a|r/[e_1,e_2]}} \frac{1}{a^4} \ll \frac{1}{e_1^2} \sum_{e_2|r} \frac{(e_1,e_2)^4}{e_2^2}.
$$
We now observe that 
\begin{align*}
\frac{1}{e_1^2}\sum_{e_2|r} \frac{(e_1,e_2)^4}{e_2^2} &= \frac{1}{e_1^2}\prod_{\ss{p^k || r \\ p^\nu || e_1 \\ 0 \leq \nu \leq k}} \left(\sum_{0 \leq j \leq \nu} p^{2j} + p^{4\nu} \sum_{\nu < j \leq k} \frac{1}{p^{2j}}\right) \\
&\ll \frac{1}{e_1^2}\prod_{\ss{p^\nu || e_1 \\ \nu \geq 1}}\left(p^{2\nu} \left(1-\frac{1}{p^2}\right)^{-1} + p^{2\nu - 2} \left(1-\frac{1}{p^2}\right)^{-1}\right) = \prod_{p|e_1} \left(1+\frac{2}{p^2-1}\right) \ll 1.
\end{align*}
Applying this bound for each $e_1|r$, we find that 
$$
\mc{R}_2 \ll  \sum_{e_1|r} \sum_{\ss{1 \leq j_1 < d \\ (j_1,r) = e_1}} \sg_{j_1} = \Sigma_{\chi}(x).
$$
Substituting our bounds for $\mc{R}_1$ and $\mc{R}_2$ into \eqref{eq:lowGCDpart}, we obtain the upper bound
$$
\frac{r\eta}{d}  \Sigma_{\chi}(x)(\mc{R}_1 + \mc{R}_2) \ll \frac{r\eta}{d} \Sigma_{\chi}(x)^2 \left(\frac{d}{r}+1\right) \ll \eta \Sigma_{\chi}(x)^2.
$$
Gathering all of the bounds together and selecting $M = \eta^{-1}$, we obtain
$$
\Delta \ll \left(\eta + \frac{1}{M}\right) \Sigma_{\chi}(x)^2 \ll \eta \Sigma_{\chi}(x)^2.
$$
Thus, by Chebyshev's inequality we see that, with $O(\eta^{1/2} \log(1/\eta) r)$ exceptions $1 \leq \ell \leq r$, whenever $\sg_j(x) \leq (r\eta/d) \Sigma_{\chi}(x)$ for all $1 \leq j \leq d-1$, 
$$
\mb{D}(\chi^{g\ell},1;x)^2 \geq 8 \sum_{1 \leq j \leq d-1} \|\frac{j \ell}{r}\|^2 \sg_j \geq \left(\frac{2}{3} + O\left(\frac{1}{\sqrt{\log(1/\eta)}}\right)\right) \Sigma_{\chi}(x).
$$
In particular, provided $\eta$ is small enough, we deduce that $\mb{D}(\chi^{g\ell},1;x)^2 \geq \tfrac{1}{2}\Sigma_{\chi}(x)$ for \emph{some} $1\leq \ell \leq r$. 
\end{proof}

\subsection{Large $\sg_j$ case}
Suppose next that there is a $1\leq j_0 \leq d-1$ for which $\sg_{j_0}(x)$ is ``large'', in contrast to the previous subsection.
Our objective is to give a strengthening of \cite[Prop. 4.5]{ManHighOrd}, in this case.
\begin{prop}\label{prop:largekBd}
There exists an absolute constant $c > 0$ such that the following holds. Suppose there is $1 \leq j_0 \leq d-1$ such that $\sg_{j_0}(x) > \tfrac{\eta}{g} \Sigma_{\chi}(x)$, and that moreover
$$
m^2\log(1/\e) < c \Sigma_{\chi}(x).
$$
Then for any parameters $z \geq 10$ and $k \in \mb{N}$ satisfying 
$$
10 \leq z \leq \exp\left(\tfrac{\log x}{3k}\right), \quad 1 \leq k \leq \sqrt{\sg_{j_0}(z)},
$$ 
for any $0 < \rho < 1-2/\pi$ and any $1 \leq \ell \leq r$,
\begin{align*}
\left(\left\|\frac{j_0\ell}{r}\right\| + O\left(\frac{m^2\log(1/\e)}{\sg_{j_0}(z)}\right)\right)^k\left|\frac{1}{x}\sum_{n \leq x} \chi^{g\ell}(n)\right| \ll k!\left((8\sg_{j_0}(z))^{-k/2} + \left(\frac{\log z}{\log x}\right)^\rho \sg_{j_0}(z)^{-1} + x^{-1/6}\right).
\end{align*}
\end{prop}
\begin{proof}
Define
$$
\mc{T}_{k,z} := \{p_1\cdots p_k : \, p_i \leq z \text{ and } \chi(p_i) = \zeta^{j_0} \text{ for all } 1 \leq i \leq k\}.
$$ 
For each $t \in \mc{T}_{k,z}$ and $n \in \mb{N}$ define\footnote{Note that when $t = p$ is prime this reduces to the ``mean $0$'' function
$f_p(n) = 1_{p|n} - \frac{1}{p}$.}
$$
f_t(n) 
:= \sum_{ab = t} \frac{\mu(b)}{b} 1_{a|n}.
$$
Let $1 \leq \ell \leq r$. Select $y_{g\ell,0}$ to be a maximiser for $\max_{|y| \leq 2 \log x} |F_{g\ell}(1+iy)|$, where for $s \in \mb{C}$ with $\text{Re}(s) > 0$ we write
$$
F_{g\ell}(s) := \prod_{p \leq x} \left(1-\frac{\chi^{g\ell}(p)}{p^s}\right)^{-1}.
$$
We then define 
$$
y_{g\ell} := \begin{cases} 0 &\text{ if } |y_{g\ell,0}| > \tfrac{1}{2}\log x, \\ y_{g\ell,0} &\text{ otherwise.} \end{cases}
$$ 
Now, by Proposition \ref{prop:shortStruc}, whenever $1 \leq \ell \leq r$ we have
$$
\mb{D}(\chi^{g\ell},1;x)^2 \leq 200m^2\log(1/\e).
$$
Since, when $y_{g\ell} \neq 0$, 
$$
\exp\left(-\mb{D}(\chi^{g\ell},n^{iy_{g\ell}};x)^2\right) \asymp
\frac{|F_{g\ell}(1+iy_{g\ell})|}{\log x} = \max_{|y| \leq 2\log x} \frac{|F_{g\ell}(1+iy)|}{\log x} \asymp \exp\left(-\min_{|y| \leq 2\log x}\mb{D}(\chi^{g\ell},n^{iy};x)^2\right),
$$
it follows that
$$
\mb{D}(\chi^{g\ell}, n^{iy_{g\ell}};x)^2 = \min_{|y| \leq 2 \log x} \mb{D}(\chi^{g\ell},n^{iy};x)^2 + O(1).
$$
Using this minimal property when $y_{g\ell} \neq 0$, we deduce that when $\e$ is sufficiently small,
$$
\mb{D}(1,n^{iy_{g\ell}};x)^2 \leq 2(\mb{D}(\chi^{g\ell},1;x) + \mb{D}(\chi^{g\ell},n^{iy_{g\ell}};x)) \leq 4\mb{D}(\chi^{g\ell},1;x)^2 + O(1) \leq 1000 m^2\log(1/\e),
$$
a bound which also trivially holds when $y_{g\ell} = 0$. From this and Lemma \ref{lem:repelChar} we deduce that there is $C \geq 1$ such that
\begin{equation}\label{eq:ylBd}
|y_{g\ell}| \ll \e^{-Cm^2}/\log x.
\end{equation}
Next, define $h_{\ell}(n) := \chi^{g\ell}(n) n^{-iy_{g\ell}}$, and consider the sum
$$
M_{h_{\ell}}(x;t) := \frac{1}{x}\sum_{n \leq x} h_{\ell}(n)f_t(n), \quad t \in \mc{T}_{k,z}.
$$
We will estimate the sum $\sum_{t \in \mc{T}_{k,z}} M_{h_{\ell}}(x;t)$ in two ways.  \\
First, we relate this sum to the mean value of $h_{\ell}$. To begin with, observe that since $h_{\ell}$ is completely multiplicative, 
$$
M_{h_{\ell}}(x;t) = \frac{1}{x}\sum_{n \leq x} h_{\ell}(n) \sum_{ac = t} \frac{\mu(c)}{c}1_{a|n} = \sum_{ac = t} \frac{\mu(c)h_{\ell}(a)}{ac} \cdot \frac{a}{x} \sum_{m \leq x/a}  h_{\ell}(m) = \frac{1}{t} \sum_{ac = t} \mu(c)h_{\ell}(a) \cdot \frac{a}{x} \sum_{m \leq x/a} h_{\ell}(m).
$$
For each $a |t$ the Lipschitz bound \cite[Thm. 4]{GSDecay} yields
$$
\frac{a}{x} \sum_{m \leq x/a} h_{\ell}(m) = \frac{1}{x} \sum_{m \leq x} h_{\ell}(m) + O\left(\left(\frac{\log (3a)}{\log x}\right)^{1-2/\pi} \log\left(\frac{\log x}{\log(3a)}\right)^2\right).
$$
It follows that
\begin{align*}
M_{h_{\ell}}(x;t) = \left(\frac{1}{t} \sum_{ac = t} \mu(c)h_{\ell}(a) \right) \frac{1}{x} \sum_{n \leq x} h_{\ell}(n) + O(\delta_t(x)),
\end{align*}
where $\delta_t$ is defined via
$$
\delta_t(x) := \frac{1}{t} \sum_{ac = t} \mu^2(c) \left(\frac{\log(3a)}{\log x}\right)^{1-2/\pi} \log\left(\frac{\log x}{\log (3a)}\right)^2.
$$
Now let
\begin{align*}
A_{k,z} &:= \sum_{t \in \mc{T}_{k,z}} \frac{1}{t}\sum_{ac = t} \mu(c)h_{\ell}(a) = \sum_{ac \in \mc{T}_{k,z}} \frac{\mu(c)}{c} \frac{\chi^{g\ell}(a)}{a^{1+iy_{g\ell}}}.
\end{align*}
Then we have
\begin{align}
A_{k,z} \frac{1}{x} \sum_{n \leq x} h_{\ell}(n) &= \sum_{t \in \mc{T}_{k,z}} M_{h_{\ell}}(x;t) + O\left(\sum_{t \in \mc{T}_{k,z}} \delta_t(x)\right). \label{eq:AkzExp}
\end{align}
We next obtain an asymptotic estimate for $A_{k,z}$. Since $ac \in \mc{T}_{k,z}$ if and only if there is $0 \leq \nu \leq k$ such that $a\in \mc{T}_{\nu,z}$ and $c \in \mc{T}_{k-\nu,z}$, taking into account the number of representations of this shape we deduce that
\begin{align}
A_{k,z} &= \sum_{0 \leq \nu \leq k} \left(\sum_{c \in \mc{T}_{k-\nu,z}} \frac{\mu(c)}{c}\right) \left(\sum_{a \in \mc{T}_{\nu,z}} \frac{\chi^{g\ell}(a)}{a^{1+iy_{g\ell}}}\right) \nonumber \\
&= \sum_{0 \leq \nu \leq k} \frac{1}{(k-\nu)! \nu !}\left(\sum_{\ss{p_1,\ldots,p_{k-\nu} \leq z \\ \chi(p_i) = \zeta^{j_0} \\ p_i \text{ distinct}}} \frac{(-1)^{k-\nu}}{p_1\cdots p_{k-\nu}}\right) \left(\sum_{\ss{p \leq z \\ \chi(p) = \zeta^{j_0}}} \frac{\chi^{\ell}(p)p^{-iy_{g\ell}}}{p}\right)^{\nu} \nonumber \\
&= \frac{1}{k!}\sum_{0 \leq \nu \leq k} \binom{k}{\nu} \zeta^{j_0g\ell\nu}(-1)^{k-\nu} \left(\sum_{\ss{p_1,\ldots,p_{k-\nu} \leq z \\ \chi(p_i) = \zeta^{j_0} \\ p_i \text{ distinct}}} \frac{1}{p_1\cdots p_{k-\nu}}\right)\left(\sum_{\ss{p \leq z \\ \chi(p) = \zeta^{j_0}}} \frac{p^{-iy_{g\ell}}}{p}\right)^{\nu}.
\label{eq:splitAC}
\end{align}
Dispensing with the distinctness condition in the first bracketed expression in \eqref{eq:splitAC}, we get 
\begin{align*}
\sum_{\ss{p_1,\ldots,p_{k-\nu} \leq z \\ \chi(p_i) = \zeta^{j_0} \\ p_i \text{ distinct}}} \frac{1}{p_1\cdots p_{k-\nu}} &= \sum_{\ss{p_1,\ldots,p_{k-\nu} \leq z \\ \chi(p_i) = \zeta^{j_0}}} \frac{1}{p_1\cdots p_{k-\nu}} + O\left(\sum_{1 \leq i < j \leq k-\nu} \sum_{\ss{p_1,\ldots,p_{k-\nu} \leq z \\ p_i = p_j \\ \chi(p_\ell) = \zeta^{j_0}}} \frac{1}{p_1\cdots p_{k-\nu}} \right) \\
&= \sg_{j_0}(z)^{k-\nu} \left(1+O\left(k^2\sg_{j_0}(z)^{-2}\right)\right),
\end{align*}
for each $0 \leq \nu \leq k-2$. Setting
$$
\Theta_{j_0,\ell}(z) := \frac{1}{\sg_{j_0}(z)} \sum_{\ss{p \leq z \\ \chi(p) = \zeta^{j_0}}} \frac{p^{-iy_{g\ell}}}{p},
$$
we may rewrite our expression for $A_{k,z}$ in \eqref{eq:splitAC} as
\begin{align*}
A_{k,z} &= \left(1+O(k^2\sg_{j_0}(z)^{-2})\right)\frac{\sg_{j_0}(z)^k}{k!}\sum_{0 \leq \nu \leq k} \binom{k}{\nu} \left(\zeta^{j_0g\ell} \Theta_{j_0,\ell}(z)\right)^\nu (-1)^{k-\nu} \\
&= \left(1+O(k^2\sg_{j_0}(z)^{-2})\right)\frac{\sg_{j_0}(z)^k}{k!} \left(\zeta^{j_0g\ell} \Theta_{j_0,\ell}(z)-1\right)^k. 
\end{align*}
Next, we estimate $\Theta_{j_0,\ell}(z)$. Set $z_0 := \min\{z,e^{1/|y_{g\ell}|}\}$. By Mertens' theorem and \eqref{eq:ylBd}, we find
\begin{align*}
\sum_{\ss{p \leq z \\ \chi(p) = \zeta^{j_0}}} \frac{|1-p^{-iy_{g_\ell}}|}{p} \ll \sum_{p \leq e^{1/|y_{g\ell}|}} \frac{|y_{g\ell}|\log p}{p} + \sum_{z_0 < p \leq z} \frac{1}{p} \ll 1 + \log\left(|y_{g\ell}|\log z\right) \ll m^2\log(1/\e).
\end{align*}
We deduce therefore that
\begin{align} \label{eq:twistedSumj0}
\Theta_{j_0,\ell}(z) = 1+ O\left(\frac{m^2\log(1/\e)}{\sg_{j_0}(z)}\right),
\end{align}
so that we may finally conclude that
\begin{align*}
A_{k,z} &= \left(1+O(k^2\sg_{j_0}(z)^{-2})\right)\frac{\sg_{j_0}(z)^k}{k!} \left(\zeta^{j_0g\ell}-1 + O\left(\frac{m^2\log(1/\e)}{\sg_{j_0}(z)}\right)\right)^k.
\end{align*}
We now obtain an upper bound for the right-hand side in \eqref{eq:AkzExp}, beginning with the average value of $M_{h_{\ell}}(x;t)$. 
By the Cauchy-Schwarz inequality,
\begin{align}
\left|\sum_{t \in \mc{T}_{k,z}} M_{h_{\ell}}(x;t)\right| &= \left|\frac{1}{x}\sum_{n \leq x} h_{\ell}(n) \left(\sum_{t \in \mc{T}_{k,z}} f_t(n)\right)\right| \leq \left( \frac{1}{x}\sum_{n \leq x} \left|\sum_{t \in \mc{T}_{k,z}} f_t(n)\right|^2\right)^{1/2} =  \left(\sum_{t_1,t_2 \in \mc{T}_{k,z}} \frac{1}{x}\sum_{n \leq x} f_{t_1}(n)f_{t_2}(n)\right)^{1/2}. \label{eq:CSMgxt}
\end{align}
Next, fixing $t_1,t_2 \in \mc{T}_{k,z}$ for the moment we observe that
\begin{align*}
\sum_{n \leq x} f_{t_1}(n)f_{t_2}(n) &= \sum_{a_1c_1 = t_1} \sum_{a_2c_2 = t_2} \frac{\mu(c_1)\mu(c_2)}{c_1c_2} \sum_{\ss{n \leq x \\ [a_1,a_2]|n}} 1 =  \sum_{a_1c_1 = t_1} \sum_{a_2c_2 = t_2} \frac{\mu(c_1)\mu(c_2)}{c_1c_2} \left(\frac{x}{[a_1,a_2]} + O(1)\right) \\
&= \frac{x}{t_1t_2} \sum_{a_1c_1 = t_1}\sum_{a_2c_2 = t_2} \mu(c_1)\mu(c_2) (a_1,a_2) + O(d(t_1)d(t_2)),
\end{align*}
where $d(t)$ is the divisor function. For each pair of divisors $a_1c_1 = t_1$, notice that\footnote{Given $n \in \mb{N}$ and a prime $p$ we write $\nu_p(n)$ to be the maximal power $\nu \geq 0$ such that $p^\nu | n$.}
\begin{equation}\label{eq:muGCD}
\sum_{a_2c_2 = t_2} \mu(c_2)(a_1,a_2) = \prod_{p^\nu || t_2} \left(p^{\min\{\nu,\nu_p(a_1)\}} - p^{\min\{\nu-1,\nu_p(a_1)\}}\right),
\end{equation}
which is non-zero precisely when 
$$
\nu_p(a_1) \geq \nu \text{ for all } p^\nu ||t_2, \text{ i.e., } t_2|a_1. 
$$
In this case, the left-hand side of \eqref{eq:muGCD} is precisely $\phi(t_2) 1_{t_2|a_1}$. Since $a_1|t_1$ we get $t_2|t_1$. Making the change of variables $a_1 = t_2A_1$, we find
\begin{align*}
\sum_{n \leq x} f_{t_1}(n)f_{t_2}(n) &= \frac{x\phi(t_2)}{t_1t_2} 1_{t_2|t_1} \sum_{\ss{A_1c_1 = t_1/t_2}}\mu(c_1) + O(d(t_1)d(t_2)) \\
&= \frac{x\phi(t_2)}{t_1t_2} 1_{t_2=t_1} + O(d(t_1)d(t_2)) = \frac{x\phi(t)}{t^2}1_{t_1 = t_2 = t} + O(d(t_1)d(t_2)).
\end{align*}
Substituting this back into \eqref{eq:CSMgxt}, using $d(t) \leq 2^k$ for all $t \in \mc{T}_{k,z}$ and $|\mc{T}_{k,z}| \leq \pi(z)^k$, we obtain
\begin{align*}
\left|\sum_{t \in \mc{T}_{k,z}} M_{h_{\ell}}(x;t)\right| &\ll \left(\sum_{t \in \mc{T}_{k,z}} \frac{\phi(t)}{t^2}\right)^{1/2} + \left(\frac{1}{x}\sum_{t_1,t_2 \in \mc{T}_{k,z}} d(t_1)d(t_2)\right)^{1/2} \\
&\ll \sg_{j_0}(z)^{k/2} + \frac{(2\pi(z))^k}{\sqrt{x}}.
\end{align*}
Finally, we estimate the contribution to \eqref{eq:AkzExp} from $\delta_t(x)$. Applying the trivial bound $a \leq t$ in the sum defining $\delta_t(x)$, we get
\begin{align*}
\delta_t(x) \leq \frac{d(t)}{t} \left(\frac{\log t}{\log x}\right)^{1-2/\pi} \log\left(\frac{\log x}{\log t}\right)^2 \ll \frac{d(t)}{t} \left(\frac{\log t}{\log x}\right)^{\rho},
\end{align*}
for any $0 < \rho < 1-2/\pi$. Using the simple inequality
$$
(a_1+ \cdots + a_m)^{\rho} \leq a_1^{\rho} + \cdots + a_m^{\rho}, \quad a_j \in (0,\infty),
$$
together with Mertens' theorem and partial summation, we deduce that
\begin{align*}
\sum_{t \in \mc{T}_{k,z}} \delta_t(x) &\leq \frac{1}{(\log x)^\rho} \sum_{t \in \mc{T}_{k,z}} \frac{d(t)}{t} \left(\sum_{p^\nu || t} \log p^\nu\right)^\rho \\
&\leq \frac{1}{(\log x)^\rho} \sum_{\ss{p \leq z \\ \nu \geq 1}} \frac{(\nu+1)(\log p^\nu)^\rho}{p^\nu}\sum_{\ss{t' \in \mc{T}_{k-\nu,z} \\ p \nmid t'}} \frac{d(t')}{t'} \\
&\ll \left(\frac{\log z}{\log x}\right)^\rho (2\sg_{j_0}(z))^{k-1}.
\end{align*}
Consequently, when $km \leq \eta \sqrt{\sg_{j_0}(z)}$ for $\eta > 0$ small enough, we find that
\begin{align*}
\frac{1}{x} \left|\sum_{n \leq x} h_{\ell}(n)\right| &\ll \frac{\sg_{j_0}(z)^k}{|A_{k,z}|} \left(\sg_{j_0}(z)^{-k/2} + 2^k\left(\frac{\log z}{\log x}\right)^\rho \sg_{j_0}(z)^{-1} + \frac{(2\pi(z))^k}{\sqrt{x}}\right).
\end{align*}
Recall that $h_{\ell}(n) = \chi^{g\ell}(n)n^{-iy_{g\ell}}$. Recalling the bound \eqref{eq:ylBd} and applying \cite[Lem. 7.1]{GSDecay}, we obtain
\begin{align}
\left|\frac{1}{x}\sum_{n \leq x} \chi^{g\ell}(n)\right| &= (1+ |y_{g\ell}|)\left|\frac{1}{x}\sum_{n \leq x} h_{\ell}(n)\right| + O\left(\frac{\exp\left(\mb{D}(\chi^{g\ell}, n^{iy_{g\ell}};x) \sqrt{(2+o(1))\log\log x}\right)}{\log x}\right) \nonumber\\
&\ll \left|\frac{1}{x} \sum_{n \leq x} h_{\ell}(n)\right| + \frac{1}{\sqrt{\log x}} \nonumber\\
&\ll \frac{\sg_{j_0}(z)^k}{|A_{k,z}|}\left(\sg_{j_0}(z)^{-k/2} + 2^k\left(\frac{\log z}{\log x}\right)^\rho \sg_{j_0}(z)^{-1} + \frac{(2\pi(z))^k}{\sqrt{x}}\right). \label{eq:RHSinProp42}
\end{align}
whenever $m^2 \log(1/\e) < c \log\log x$ for $c > 0$ small enough. In light of the lower bound 
$$
|1-\zeta^{j_0g\ell}| = 2|\sin(\pi j_0g\ell/d)| \geq 4\|j_0 g\ell/d\| = 4\|j_0\ell/r\|,
$$
we obtain that
$$
|A_{k,z}| \gg \frac{4^k\sg_{j_0}(z)^k}{k!} \left(\left\|\frac{j_0\ell}{r}\right\| + O\left(\frac{m^2\log(1/\e)}{\sg_{j_0}(z)}\right)\right)^k.
$$
Finally, as $z^k \leq x^{1/3}$, the last error term in \eqref{eq:RHSinProp42} is $O(x^{-1/6})$. The claim now follows upon rearranging.
\end{proof}
The bound in Proposition \ref{prop:largekBd} is only efficient provided we can obtain a lower bound for $\|j_0\ell/r\|$ for many $\ell$. The purpose of the next result is to bound the number of $\ell$ for which $\|j_0\ell/r\|$ is small.
\begin{prop} \label{prop:discBound}
With the above notation, there is an absolute constant $C > 0$ such that if
$$
\theta_0 := C\left(\frac{m^2\eta^{-1}\log(1/\e)}{\log\log d}\right)^{1/2}.
$$
then at least one of the following is true:
\begin{enumerate}[(i)]
\item $|S_{\chi}(x)| \leq \e x$, and
\item $(j_0,r) \leq \theta_0 r$.
\end{enumerate}
Moreover, in the second case we find that for any $\theta \in [\theta_0,1/2]$,
$$
|\{1 \leq \ell \leq r : \|j_0 \ell/r\| \leq \theta\}| \ll \theta r.
$$
\end{prop}
\begin{proof}
Assume (i) fails, so that $|S_{\chi}(x)| \geq \e x$, i.e., $1 \in \mc{C}_d(\e)$. By Proposition \ref{prop:shortStruc}, $g\ell \equiv 1 \pmod{d}$ for some $g|d$, whence $g = 1$, and 
$$
\mb{D}(\chi,1;x)^2 \leq 200 m^2 \log(1/\e).
$$
Combining this with \eqref{eq:toFE} and $\ell = 1$, we find that
\begin{equation}\label{eq:j0SigmaBd}
200 m^2 \log(1/\e) \geq 8\sum_{1 \leq j \leq d-1} \left\|\frac{j}{d}\right\|^2 \sg_j \geq 8\left\|\frac{j_0}{d}\right\|^2 \sg_{j_0} \geq 8\eta\left\|\frac{j_0}{d}\right\|^2 \Sigma_{\chi}(x).
\end{equation}
Let $\gamma := (j_0,d)$ and put $\tilde{j}_0 := j_0/\gamma$, $D := d/\gamma$. Since 
$$
\|j_0/d\| = \|\tilde{j}_0/D\| \geq 1/D = \gamma/d,
$$ 
it follows from Proposition \ref{prop:LowBdSigma} that there is a constant $C >0$ such that
$$
\gamma \leq d \left(\frac{25 m^2 \eta^{-1}\log(1/\e)}{\Sigma_{\chi}(x)}\right)^{1/2} \leq Cd\left(\frac{m^2 \eta^{-1}\log(1/\e)}{\log\log d}\right)^{1/2} = \theta_0 d.
$$
Since $r = d$, $(j_0,r)/r = (j_0,d)/d$ and (ii) follows. \\
Next, write $R := r/(r,j_0)$ and $J_0 := j_0/(r,j_0)$. Dividing $\ell = m R + L$ with $0 \leq m \leq r/R-1$ and $1 \leq L \leq R$, observe that
$$
\|j_0 \ell/r\| = \|J_0 \ell/R\| = \|J_0 L/R\|.
$$
Let $\theta \in [\theta_0,1/2]$. As $(J_0,R) = 1$, taking $V := R/2$ and applying the Erd\H{o}s-Tur\'{a}n inequality \cite[I.6.15]{Ten}, we obtain
\begin{align*}
\left|\left|\left\{1 \leq L \leq R : \, \left\|\frac{J_0 L}{R}\right\| \leq \theta\right\}\right| - 2\theta R\right| &\ll R\left(\frac{1}{V} + \sum_{1 \leq \nu \leq V} \frac{1}{\nu} \left|\frac{1}{R}\sum_{\ss{1 \leq L \leq R}} e\left(\frac{\nu J_0 L}{R}\right)\right|\right) \\
&= R\left(\frac{1}{V} + \sum_{\ss{1 \leq \nu \leq V \\ R|J_0\nu}} \frac{1}{\nu}\right)  \ll 1.
\end{align*}
It follows that
$$
|\{1 \leq \ell \leq r : \|j_0 \ell/r\| \leq \theta\}| = \frac{r}{R} \cdot \left(2\theta R + O(1)\right) = r\left(2\theta + O(1/R)\right).
$$
As $R  = r/(j_0,r) > \theta^{-1}$, we obtain
$$
|\{1 \leq \ell \leq r : \|j_0 \ell/r\| \leq \theta \}| \ll \theta r,
$$
as claimed.
\end{proof}
\begin{proof}[Proof of Theorem \ref{thm:impShort}]
Since the bound in the theorem is otherwise trivial, we may assume that $d$ is larger than any fixed constant. Hence, we may assume that $x$ is also larger than any fixed constant, given the constraint $x > q^\delta$. \\
Let $\tau \in (0,1/2)$ and let $\e := \Sigma_{\chi}(x)^{-1/(6+\tau)}$ and suppose throughout that $|S_{\chi}(x)| \geq \e x$. Assume for the sake of contradiction that $|\mc{C}_d(\e)| \geq \e d$. \\
By Proposition \ref{prop:shortStruc} we find integers $1 \leq m \leq \e^{-2}$ and $1 \leq g \leq \e^{-1}$ such that
$$
\max_{1 \leq \ell \leq d/g} \mb{D}(\chi^{g\ell},1;x)^2 \ll m^2 \log(1/\e).
$$ 
Fix $\eta = \theta = \e^{\tau/30}$, so that since 
$$
\theta_0 \ll \left(\frac{\e^{-4 - \tau/30} \log(1/\e)}{\e^{-(6+\tau)}}\right)^{1/2} \ll \e^{1/2},
$$
we have $\theta \in [\theta_0,1/2]$, when $\e$ is sufficiently small. \\
Let $1 \leq j_0 \leq d-1$ be an index satisfying 
$$
\sg_{j_0}(x) = \max_{1 \leq j \leq d-1} \sg_j(x).
$$ 
If $\sg_{j_0}(x) \leq \tfrac{\eta}{g}\Sigma_{\chi}(x)$ then by Proposition \ref{prop:smallSig} we 
can find $1 \leq \ell \leq d/g$ such that
$$
\mb{D}(\chi^{g\ell},1;x)^2 \geq \tfrac{1}{2}\Sigma_{\chi}(x).
$$
Comparing these bounds using Proposition \ref{prop:LowBdSigma} and $m \leq \e^{-2}$, we find
$$
\e^{-6-\tau} = \Sigma_{\chi}(x) \ll m^2 \log(1/\e) \leq \e^{-4}\log(1/\e),
$$
which is a\footnote{We emphasise that this part of the argument is independent of the assumption $|S_{\chi}(x)| \geq \e x$.} contradiction.\\
Next, suppose $\sg_{j_0}(x) > \tfrac{\eta}{g}\Sigma_{\chi}(x)$. Since $1 \in \mc{C}_d(\e)$ we deduce that $g = 1$, and $d = r$. Let $M \geq 2$ be a parameter to be chosen later, and set $z := x^{1/M}$. We have the crude lower bound
$$
\sg_{j_0}(z) \geq \sg_{j_0}(x) - \sum_{x^{1/M} < p \leq x} \frac{1}{p} \geq \eta \Sigma_{\chi}(x) - \log M - O(1).
$$
Assume henceforth that $M$ is chosen so that $\log M \leq \frac{\eta}{3} \Sigma_{\chi}(x)$. Thus, when $d$ is large enough we have 
$$
\sg_{j_0}(z) \geq \frac{\eta}{2} \Sigma_{\chi}(x).
$$
We first establish the existence of $\ell$ with $\|j_0\ell/d\| > \theta$.  
Assume for the sake of contradiction that
$$
\max_{\tilde{\ell} \in \mc{C}_d(\e)} \|j_0\tilde{\ell}/d\| \leq \theta/m.
$$
Since $\mb{Z}/\ell \mb{Z} = m\mc{C}_d(\e)$, if 
$$
\ell \equiv a_1 + \cdots + a_m \pmod{d}, \quad a_i \in \mc{C}_d(\e)
$$
then the triangle inequality implies that for every $1 \leq \ell \leq d$,
$$
\left\|\frac{j_0 \ell}{d} \right\| \leq \sum_{1 \leq i \leq m } \left\|\frac{j_0 a_i}{d}\right\| \leq m \cdot \frac{\theta}{m} \leq \theta.
$$
On the other hand, since $|S_{\chi}(x)| \geq \e x$, Proposition \ref{prop:discBound} shows that 
$$
|\{1 \leq \ell \leq d : \|j_0 \ell/d\| \leq \theta\}| \ll \theta d,
$$
which is a contradiction.
In particular, there must exist $\ell \in \mc{C}_d(\e)$ for which $\|j_0\ell/d\| > \theta/m$. Note that by our parameter choices, we have
$$
\frac{m^2\log(1/\e)}{\sg_{j_0}(z)} < \frac{2m^2\eta^{-1} \log(1/\e)}{\Sigma_{\chi}(x)} \ll \e^{6+\tau -4 - \tau/2} = \e^{2 + \tau/2} \leq \e^{\tau/4} \frac{\theta}{m} < \e^{\tau/4} \left\|\frac{j_0\ell}{d}\right\|. 
$$
Taking $\rho = 1/4$ and applying Proposition \ref{prop:largekBd} we thus find that if 
$$
1 \leq k \leq \min\left\{\frac{M}{3}, \sqrt{\eta \Sigma_{\chi}(x)}\right\}
$$
then for any $1 \leq \ell \leq r$ for which $r \nmid j_0\ell$ we have
\begin{align} \label{eq:glUppBd}
\frac{1}{x}|S_{\chi^{\ell}}(x)| \ll k! \left\|\frac{j_0\ell}{r}\right\|^{-k} \left(\left(\frac{1}{\eta \Sigma_{\chi}(x)}\right)^{k/2} +\frac{1}{\eta M^{1/4}\Sigma_{\chi}(x)} + x^{-1/6}\right).
\end{align}
By Proposition \ref{prop:shortStruc}, $\mc{C}_d(\e) \subseteq m\mc{C}_d(\e) = \mb{Z}/d\mb{Z}$. 
We deduce using $k! \leq k^k$ that
$$
\e \leq \frac{1}{x}|S_{\chi^{g\ell}}(x)| \ll \left(\left(\frac{(km)^2}{\eta \theta^2 \Sigma_{\chi}(x)}\right)^{k/2} + \frac{(2km/\theta)^k }{\eta M^{1/4}\Sigma_{\chi}(x)}\right).
$$
As $m \leq \e^{-2}$, we get that
\begin{equation}\label{eq:contraShort}
\e \ll \left(\left(\frac{k^2\e^{-4}}{\eta \theta^2\Sigma_{\chi}(x)}\right)^{k/2} + \frac{(2k/(\e^2\theta))^k}{\eta M^{1/4}\Sigma_{\chi}(x)}\right).
\end{equation}
Set now $\eta = \theta = \e^{\tau/30}$, choose $M$ so that $\log M = \frac{\eta}{3} \Sigma_{\chi}(x)$ and select 
$$
k := \left \lfloor \frac{10}{\tau}\sqrt{\e^{4+\tau/2} \Sigma_{\chi}(x)}\right \rfloor.
$$ 
We then find that if $d$ is sufficiently large relative to $\tau$ then
\begin{align*}
\left(\frac{k^2\e^{-4}}{\eta \theta^2 \Sigma_{\chi}(x)}\right)^{k/2} &\leq \left(100\tau^{-2} \e^{2\tau/5}\right)^{\frac{5}{\tau}} \ll_{\tau} \e^{2}, 
\end{align*}
on the one hand, and also
\begin{align*}
\frac{(2k/(\e^2\theta))^k \e^{-1}}{\eta M^{1/4}\Sigma_{\chi}(x)} &\ll \frac{1}{\e^{31/30} \Sigma_{\chi}(x)} \exp\left(-\frac{\e^{31/30}}{12}\Sigma_{\chi}(x) + 10\frac{\e^{5/2}}{\tau}\Sigma_{\chi}(x)^{1/2} \log \Sigma_{\chi}(x)\right) \\
&\ll \Sigma_{\chi}(x)^{-2/3}e^{-\frac{\e^2}{24} \Sigma_{\chi}(x)} \ll \Sigma_{\chi}(x)^{-100} \\
&\ll \e^{600},
\end{align*}
on the other. It follows that \eqref{eq:contraShort} yields a contradiction. \\
We conclude, therefore, that $|\mc{C}_d(\e)| \leq \e d$, and therefore
$$
\frac{1}{d}\sum_{1 \leq \ell \leq d} |S_{\chi^\ell}(x)|^2 \leq \e^2 + \frac{|\mc{C}_d(\e)|}{d} \leq \e.
$$
As $\e = \Sigma_{\chi}(x)^{-1/(6+\tau)} \ll (\log\log d)^{-\tfrac{1}{6}+\tau}$ by Proposition \ref{prop:LowBdSigma}, the proof of the theorem follows.
\end{proof}
\section{Proof of Theorem \ref{thm:paucity}} \label{sec:pauc}
In this section, we present the proof of Theorem \ref{thm:paucity}. For convenience, we introduce the notation
$$
M_{d,\chi}(x) := \max_{\alpha^d = 1} |\{n \leq x : \chi(n) = \alpha\}|.
$$
We begin by establishing a couple of lemmas.
\begin{lem}\label{lem:passtoPow}
We have 
$$
M_{d,\chi}(x) \leq \min_{r|d} M_{r,\chi^{d/r}}(x).
$$
\end{lem}
\begin{proof}
It clearly suffices to prove that $M_{d,\chi}(x) \leq M_{r,\chi^{d/r}}(x)$ for each $r|d$. Thus, fix $r | d$ and let $\psi := \chi^{d/r}$. Let $\alpha$ be a $d$th order root of unity for which
$$
M_{d,\chi}(x) = |\{n \leq x : \chi(n) = \alpha\}|,
$$
and let $\beta := \alpha^{d/r}$. Then $\psi(n) = \beta$ whenever $\chi(n) = \alpha$. Since $\beta$ is a root of unity of order $r$, we obtain
$$
M_{d,\chi}(x) \leq |\{n \leq x : \psi(n) = \beta\}| \leq M_{r,\psi}(x),
$$
as claimed.
\end{proof}
\begin{lem}\label{lem:ETL2Alt}
Let $\chi$ be a character modulo $q$ of order $d$, and let $K \geq 1$. Then
$$
M_{d,\chi}(x) \leq \min\left\{\left(\frac{1}{d} \sum_{1 \leq \ell \leq d} |S_{\chi^\ell}(x)|^2\right)^{\tfrac{1}{2}}, \, \frac{x}{d} + \frac{x}{K+1} + \frac{2}{3} \sum_{1 \leq k \leq K} \frac{|S_{\chi^k}(x)|}{k}\right\}.
$$
\end{lem}
\begin{proof}
We prove each of the above bounds in sequence. The proof of the first, which is already invoked in \cite[Sec. 3.2]{ManHighOrd}, is as follows. Given any $d$th root of unity $\alpha$ we have
\begin{align}\label{eq:L2toLevel}
|\{n \leq x: \chi(n) = \alpha\}| = |\{n,m \leq x: \chi(n) = \chi(m) = \alpha\}|^{1/2} &\leq |\{n,m \leq x: \chi(n) = \chi(m)\}|^{1/2} \nonumber \\
&= \left(\frac{1}{d}\sum_{\ell = 0}^{d-1} |S_{\chi^\ell}(x)|^2 \right)^{1/2}.
\end{align}
Maximising over $\alpha$, the first alternative bound for $M_{d,\chi}(x)$ holds. \\
For the second bound,
set $N := |\{n \leq x : (n,q) = 1\}$ and for each $(n,q) = 1$ write $\chi(n) = e(\theta_n)$ for some $\theta_n \in [0,1]$. Specifying that $\chi(n) = \alpha = e(a/d)$ is equivalent to  
$$
\theta_n  \in \left[\frac{a}{d}, \frac{a+1}{d}\right).
$$
By the Erd\H{o}s-Tur\'{a}n inequality (in the form given in \cite[I.6.15]{Ten}), for any $K \geq 1$ we have
\begin{align} \label{eq:ErdTurtoLevel}
|\{n \leq x : \chi(n) = \alpha\}| &\leq \frac{N}{d} + \sup_{I \subseteq [0,1)} \left|\, |\{n \leq x : \, (n,q) = 1, \, \theta_n \in I\}| - |I|N \,\right| \nonumber\\
&\leq N\left(\frac{1}{d} + \frac{1}{K+1} + \frac{2}{3}\sum_{1\leq k \leq K} \frac{1}{k} \left|\frac{1}{N} \sum_{\ss{1 \leq n \leq x \\ (n,q) = 1}} e(k\theta_n)\right|\right) \nonumber \\
&\leq x\left(\frac{1}{d} + \frac{1}{K+1} + \frac{2}{3} \sum_{1 \leq k \leq K} \frac{1}{k}\left|\frac{1}{x} \sum_{n \leq x} \chi(n)^k\right|\right).
\end{align}
This implies the claim.
\end{proof}
\begin{proof}[Proof of Theorem \ref{thm:paucity}]
Let $c_1 > 0$ be a sufficiently small constant to be determined later. Given $1 \leq z \leq \log\log d$, write 
$$
\delta_z := \max\left\{\left(\frac{\log\log(ed_z)}{c_1 \log(ed_z)}\right)^{1/2}, (\log q)^{-c_1}\right\}, \quad d_z := \prod_{\ss{p^k||d \\ p > z}} \text{ and } r_z := d/d_z.
$$
By Lemma \ref{lem:passtoPow} we have
$$
M_{d,\chi}(x) \leq \min_{r|d} M_{r,\chi^{d/r}}(x) \leq \min_{\ss{1 \leq z \leq \log\log d \\ x > q^{\delta_z}}} M_{d_z, \chi^{r_z}}(x).
$$
For ease of notation, we write $\tilde{d} := d_{z_0}$ and $\tilde{\chi} := \chi^{r_{z_0}}$, where $1\leq z_0 \leq \log\log d$ minimises this latter upper bound, subject to the condition that $x > q^{\delta_{z_0}}$. In the remainder of the proof, we show that
\begin{equation}\label{eq:redtoz0}
M_{\tilde{d},\tilde{\chi}}(x) \leq x\left(\frac{1}{P^-(\tilde{d})} + O\left(\frac{1}{(\log\log \tilde{d})^c}\right)\right),
\end{equation}
and since $P^-(\tilde{d}) > z_0$ this provides the required upper bound. \\
Similarly to the proof of Theorem \ref{thm:impShort}, let $\e := \Sigma_{\chi}(x)^{-1/7}$
\ \\
\underline{Case 1:} Assume first that 
$\phi(q)/q < \frac{1}{(\log d)^{1/10}}$. Picking $K := \lfloor (\log d)^{1/10}\rfloor$ in \eqref{eq:ErdTurtoLevel}, together with the uniform upper bound
$$
\left|S_{\tilde{\chi}^k}(x)\right| \leq \sum_{\ss{n \leq x \\ (n,q) = 1}} 1 \ll \frac{\phi(q)}{q}x \leq \frac{x}{(\log d)^{1/5}}
$$
(using \cite{Hall} as before), we deduce that 
$$
M_{\tilde{d},\tilde{\chi}}(x) \leq x\left(\frac{1}{\tilde{d}} + \e + O\left(\frac{x}{(\log d)^{1/5}} \sum_{1 \leq k \leq (\log d)^{1/10}} \frac{1}{k}\right)\right) 
\ll \frac{x\log\log d}{(\log \tilde{d})^{1/5}},
$$
which is more than enough suffices for the bound we seek. In the remaining cases we shall assume that $\phi(q)/q > \frac{1}{(\log d)^{1/10}}$, and in particular by Proposition \ref{prop:LowBdSigma}, $\Sigma_{\chi}(x) > c \log\log d$ for some absolute constant $c > 0$.  We assume this lower bound henceforth.\\
\\
\underline{Case 2:} Assume next that $|\mc{C}_{\tilde{d}}(\e)| \leq \e \tilde{d}$. By \eqref{eq:L2toLevel}, we have
\begin{align*}
M_{\tilde{d}, \tilde{\chi}}(x) &\leq \left(\frac{1}{\tilde{d}} \sum_{\ell \notin \mc{C}_{\tilde{d}}(\e)} |S_{\tilde{\chi}^\ell}(x)|^2 + \frac{1}{\tilde{d}}\sum_{\ell \in \mc{C}_{\tilde{d}}(\e)}|S_{\tilde{\chi}^\ell}(x)|^2\right)^{1/2} \\
&\leq \left(\e^2 x^2 + x^2\frac{|\mc{C}_{\tilde{d}}(\e)|}{\tilde{d}}\right)^{1/2} \leq 2\sqrt{\e} x \ll \frac{x}{(\log\log \tilde{d})^{1/14}},
\end{align*}
which is more than sufficient. \\
\ \\
\underline{Case 3:} Assume $|\mc{C}_{\tilde{d}}(\e)| > \e \tilde{d}$ and
$\e^{-7} = \Sigma_{\tilde{\chi}}(x) \geq c \log\log \tilde{d}$. 
We consider several subcases below. \\
\ \\
\underline{Case 3.(i):} Assume that $g > 1$ in Proposition 3.1. Since $g|\tilde{d}$, we have $g \geq P^-(\tilde{d})$, and as $\mc{C}_{\tilde{d}}(\e) \subseteq \{g\ell : 1 \leq \ell < \tilde{d}/g\}$ we have that $|S_{\tilde{\chi}^k}(x)| \leq \e x$ for all $1 \leq k < g$. Applying \eqref{eq:ErdTurtoLevel} with $K = g-1$ and recalling that $g \leq \e^{-1}$, we get
$$
M_{\tilde{d},\tilde{\chi}}(x) \leq x\left(\frac{1}{\tilde{d}} + \frac{1}{g} + O\left(\e \log g\right)\right) \leq x\left(\frac{1}{P^-(\tilde{d})} + O\left(\frac{\log\log\log \tilde{d}}{(\log\log \tilde{d})^{1/7}}\right)\right),
$$
which implies the bound \eqref{eq:redtoz0} in this case.
We will assume for the remaining subcases that $g = 1$.\\
\, \\
\underline{Case 3.(ii)} If $|S_{\tilde{\chi}}(x)| > \e x$ then as $x > q^{\delta_{z_0}}$, Theorem \ref{thm:impShort} immediately implies that
$$
\frac{1}{\tilde{d}} \sum_{\ell = 0}^{\tilde{d}-1} |S_{\tilde{\chi}^\ell}(x)|^2 \ll \frac{x}{(\log\log \tilde{d})^{1/7}},
$$
and the claimed bound (with $c = 1/14$) follows from \eqref{eq:L2toLevel}.  \\
\ \\
\underline{Case 3. (iii)} If $|S_{\tilde{\chi}}(x)| \leq \e x$ and (recalling $g = 1$) $\max_{1 \leq j \leq \tilde{d}-1} \sg_j(x) \leq \eta \Sigma_{\tilde{\chi}}(x)$ then we obtain a contradiction to $|\mc{C}_{\tilde{d}}(\e)| > \e \tilde{d}$ as in the proof of Theorem \ref{thm:impShort} (as this part of the argument did not require $|S_{\tilde{\chi}}(x)| > \e x$). We deduce that $|\mc{C}_{\tilde{d}}(\e)| \leq \e \tilde{d}$, and the claim follows from Case 2 above. \\
\ \\
\underline{Case 3.(iv)} Assume that $|S_{\tilde{\chi}}(x)| \leq \e x$, $\sg_{j_0}(x) = \max_{1 \leq j \leq \tilde{d}-1} \sg_j(x) > \eta \Sigma_{\tilde{\chi}}(x)$ and (as $r = \tilde{d}/g = \tilde{d}$) $(j_0,\tilde{d}) \leq \theta \tilde{d}$, noting that here $\theta > \e^{1/2}$. From this, we deduce by Proposition \ref{prop:discBound}  
that there exists an $\ell$ such that $\|j_0\ell/\tilde{d}\| > \theta/m$. Following the proof of Theorem \ref{thm:impShort}, this was enough to deduce a contradiction to the assumption $|\mc{C}_{\tilde{d}}(\e)| > \e \tilde{d}$, and so the claim of the current proposition follows once again from Case 2. \\
\ \\
\underline{Case 3. (v)} Finally, we assume the following data:
\begin{itemize}
\item $g = 1$, so $r = \tilde{d}$
\item $|S_{\tilde{\chi}}(x)| \leq \e x$, 
\item $\sg_{j_0}(x) > \eta \Sigma_{\tilde{\chi}}(x)$
\item $\Sigma_{\tilde{\chi}}(x) > c \log\log \tilde{d}$
\item $(j_0,\tilde{d}) > \theta \tilde{d}$.
\end{itemize}
Set $R := \tilde{d}/(j_0,\tilde{d}) < \theta^{-1}$, and note that for each $1 \leq k \leq R-1$ we have
$$
\left\|\frac{j_0k}{\tilde{d}}\right\| = \left\|\frac{k (j_0/(j_0,\tilde{d}))}{R} \right\| \geq \frac{1}{R},
$$
since $j/(j_0,\tilde{d})$ is coprime to $R$. It follows from equation \eqref{eq:toFE} that for every $1 \leq k \leq R-1$,
$$
\mb{D}(\tilde{\chi}^k,1;x)^2 \geq 8 \sum_{1 \leq j \leq \tilde{d}-1} \left\|\frac{j k}{\tilde{d}}\right\|^2 \sg_j(x) \geq 8\left\|\frac{j_0k}{\tilde{d}}\right\|^2 \sg_{j_0}(x) > \frac{8\eta}{R^2} \Sigma_{\tilde{\chi}}(x)
\geq 8 \eta \theta^2 \Sigma_{\tilde{\chi}}(x).
$$
Now, either (a) $k \notin \mc{C}_{\tilde{d}}(\e)$, or else (b) $k \in m\mc{C}_{\tilde{d}}(\e)$. In case (a) we immediately have $|S_{\tilde{\chi}^k}(x)| \leq \e x$. Thus, assume that (b) holds.
From the proof of Proposition 3.1, if $t_k$ is a minimiser for $\mb{D}(\tilde{\chi}^k,n^{it};x)$ from $[-2m/\e^2,2m/\e^2]$ then
$|t_k|\log x \leq 3\e^{-64m^2}$, and therefore
$$
\mb{D}(1,n^{it_k};x)^2 = \log(1+|t_k|\log x) + O(1) \leq 64m^2 \log(1/\e) + O(1) \leq 64 \e^{-4} \log(1/\e) + O(1).
$$
Now, by the pretentious triangle inequality we have
\begin{align*}
\mb{D}(\tilde{\chi}^k,n^{it_k};x) &\geq \mb{D}(\tilde{\chi}^k,1;x) - \mb{D}(1,n^{it_{k}};x). 
\end{align*}
Given the bound $\eta,\theta > \e^{1/8}$, we have 
$$
\mb{D}(\tilde{\chi}^k,1;x) > \left(8 \eta \theta^2\right)^{1/2} \Sigma_{\tilde{\chi}}(x)^{1/2} > 2 \sqrt{2} \e^{3/16 - 7/2} \geq 2\sqrt{2}\e^{-3} > 20 \e^{-2} \log(1/\e) \geq 2 \mb{D}(1,n^{it_{k}};x).
$$
whenever $\tilde{d}$ is sufficiently large. Hence we have
$$
\mb{D}(\tilde{\chi}^k,n^{it_k};x)^2 \geq \frac{1}{4} \mb{D}(\tilde{\chi}^{k},1;x)^2 > 2\eta \theta^2 \Sigma_{\tilde{\chi}}(x) > 2c\e^{-6}. 
$$
Applying Hal\'{a}sz' theorem (with $T = m/\e^2$), we obtain, for each $1 \leq k < R$,
$$
|S_{\tilde{\chi}^k}(x)| \ll x\left(\e^2 + \mb{D}(\tilde{\chi}^k,n^{it_k};x)^2 e^{-\mb{D}(\tilde{\chi}^k,n^{it_k};x)^2}\right) \ll x\left(\e^2 + \e^{-6}\exp\left(-c\e^{-6}\right)\right) \ll \e^2 x.
$$
It follows that in either of cases (a) and (b), $|S_{\chi^k}(x)| \ll \e x$ for all $1 \leq k \leq R-1$. \\
Applying this in \eqref{eq:ErdTurtoLevel} with $K = R \leq \theta^{-1}$, we get
$$
M_{\tilde{d}, \tilde{\chi}}(x) \leq x\left(\frac{1}{\tilde{d}} + \frac{1}{R} + O\left(\e^2 \log R\right)\right) \leq x\left(\frac{1}{R} + O\left(\e \log(1/\e)\right)\right).
$$
Since $j_0 < \tilde{d}$, we have $(j_0,\tilde{d}) \leq \tilde{d}/P^-(\tilde{d})$, and therefore as $\e \gg (\log\log \tilde{d})^{-1/6}$, we obtain
$$
M_{\tilde{d}, \tilde{\chi}}(x) \leq x\left(\frac{1}{P^-(\tilde{d})} + O\left(\frac{\log\log\log \tilde{d}}{(\log\log \tilde{d})^{1/7}}\right)\right).
$$
\underline{Conclusion:} Summarising all of the above cases, we obtain, for some $c \in (0,1/14]$,
$$
M_{\tilde{d}, \tilde{\chi}}(x) \leq x\left(\frac{1}{P^-(\tilde{d})} + O\left(\frac{1}{(\log\log \tilde{d})^c}\right)\right),
$$
as claimed.
\end{proof}

\begin{proof}[Proof of Corollary \ref{cor:paucity}]
If $z \in (P^+(d) - 1, P^+(d))$ then $d_z \geq P^+(d)$. Taking $z \ra P^+(d)$ from below, and letting $\delta = \delta_{P^+(d)-1/(\log q)}$, we find that when $x > q^{\delta}$,
$$
\max_{\alpha^d = 1} \frac{1}{x}|\{n \leq x : \chi(n) = \alpha\}| \leq \left(\frac{1}{P^+(d)} + O\left(\frac{1}{(\log\log d_z)^{c_2}}\right)\right) \ll \frac{1}{(\log\log P^+(d))^{c_2}},
$$
as claimed.
\end{proof}
\section{Averaged Maximal Character Sums: Proof of Theorem \ref{thm:impMax}} \label{sec:max}
Our strategy towards proving Theorem \ref{thm:impMax} will be similar to the proof of Theorem \ref{thm:impShort}. Given $\e > 0$, recall that
$$
\mc{L}_d(\e) := \{1 \leq \ell \leq d : M(\chi^\ell) > \e \sqrt{q}\log q\}.
$$
Assuming that $|\mc{L}_d(\e)| \geq \e d$, Proposition \ref{prop:maxStruc} shows that there are positive integers $m \leq \e^{-2}$, $g \leq \e^{-1}$ and $k \leq e^{-3m}$, and a Dirichlet character $\xi \pmod{k}$ of order dividing $r = d/g$ such that
$$
\mc{L}_d(\e) \subseteq \{g \ell : 1 \leq \ell \leq r\} = m \mc{L}_d(\e),
$$
and furthermore
$$
\max_{1 \leq \ell \leq r} \mb{D}(\chi^{g\ell}, \xi^\ell;x)^2 \ll m^2 \log(1/\e).
$$
Set $\psi :=\chi^g\bar{\xi}$, so that $\psi^r$ is principal. Write $\omega = e(1/r)$ and put
\begin{align*}
\Sigma_{\psi}(q) = \sum\limits_{1 \leq j \leq r-1}\tilde{\sigma}_j(q), \text{ where } \tilde{\sigma}_j(q) = \sum\limits_{\substack{p \leq q,\\ \psi(p) = \omega^j}}\frac{1}{p}.
\end{align*}
We will study the influence of the sizes of the prime sums $\tilde{\sigma}_j(q)$ on the maximal sums $M(\chi^\ell)$. Throughout this section, fix $\eta \in (0, 1)$ to be a small parameter.

\subsection{Small $\tilde{\sigma}_j$ case}
We assume first of all that $\tilde{\sg}_j(q) \leq \tfrac{\eta}{g}\Sigma_{\psi}(q)$ for all $1 \leq j \leq r$.
\begin{prop}
Let $\eta > 0$ be sufficiently small, and suppose $\tilde{\sigma}_j(q) \leq \frac{\eta}{g}\Sigma_{\psi}(q)$ for all $1 \leq j \leq r-1$. 
Then there are elements $1\leq \ell \leq r$ such that 
\begin{align*}
    \mb{D}(\chi^{g\ell}, \xi^\ell; q)^2 \geq \frac{1}{2}\Sigma_{\psi}(q).
\end{align*}
\end{prop}
\begin{proof}
Since $\psi = \chi^g \bar{\xi}$ has order dividing $r$ and
$$
\mb{D}(\chi^{g\ell},\xi^\ell;q) = \mb{D}(\psi^\ell,1;q)
$$
for each $1 \leq \ell \leq r$, the result follows upon applying Proposition \ref{prop:smallSig} to $\psi$ in place of $\chi^g$ (as $\Sigma_{\chi}(x) \geq \Sigma_{\chi^g}(x)$ there). 
\end{proof}

\subsection{Large $\tilde{\sigma}_j$ case}
Next, we assume that $\tilde{\sg}_j(q) > \tfrac{\eta}{g} \Sigma_{\psi}(q)$ for some $j = j_0$.
\begin{prop}
Suppose there is $1 \leq j_0 \leq r-1$ such that $\tilde{\sg}_{j_0}(q) > \frac{\eta}{g}\Sigma_{\psi}(q)$. For each $1 \leq \ell \leq r$ let $1 \leq N_{\ell} \leq q$ satisfy  
$$
|L_{\psi^\ell}(N_{\ell})| = \max_{1 \leq N \leq q} |L_{\psi^\ell}(N)|.
$$ 
Then for any $1 \leq \ell \leq r$, 
\begin{align*}
    \frac{|L_{\psi^\ell}(N_\ell)|}{\log q} \ll \min\Bigg\{1, \Big\|\frac{j_0g\ell}{d}\Big\|^{-1}\frac{\log N_\ell}{(\log q)\sqrt{\tilde{\sg}_{j_0}(N_\ell)}}\Bigg\}.
\end{align*}
\end{prop}
\begin{proof}
For $g\ell \in \mc{L}_d(\e)$, by the trivial bound
\begin{align*}
\e \log q \ll |L_{\psi^\ell}(N_\ell)| \leq \log N_\ell,
\end{align*}
we have $N_\ell \gg q^\e$. We introduce the completely additive function
\begin{align*}
\Omega_{j_0}(n) 
= \sum\limits_{\substack{p^k | n, \\ \psi(p) = \omega^{j_0} }} 1.
\end{align*}
By the complete multiplicativity of $\psi$, we have
\begin{align} \label{eq:LpsiN}
\frac{L_{\psi^\ell}(N_\ell)}{\log q} 
&= \frac{1}{\log q}\sum\limits_{n \leq N_\ell}\frac{\psi^\ell(n)}{n} \cdot \frac{\Omega_{j_0}(n) + \tilde{\sg}_{j_0}(N_\ell)-\Omega_{j_0}(n)}{\tilde{\sg}_{j_0}(N_\ell)} \nonumber\\
&= \frac{1}{\tilde{\sg}_{j_0}(N_\ell) \log q} \sum\limits_{\substack{mp^k \leq N_\ell, \\ \psi(p) = \omega^{j_0}}}\frac{\psi^\ell(p)^k}{p^k}\frac{\psi^\ell(m)}{m} + O\Bigg(\frac{1}{\log q}\sum\limits_{n \leq N_\ell} \frac{|\Omega_{j_0}(n)/\tilde{\sg}_{j_0}(N_\ell) -1|}{n}\Bigg). 
\end{align}
We first estimate the error term above. By the Cauchy-Schwarz inequality,
\begin{align} \label{eq:CSbeforeTK}
\sum\limits_{n \leq N_\ell}\frac{|\Omega_{j_0}(n)/\tilde{\sg}_{j_0}(N_\ell) -1|}{n} \leq \Bigg(\sum\limits_{n \leq N_\ell}\frac{|\Omega_{j_0}(n)/\tilde{\sg}_{j_0}(N_\ell) -1|^2}{n}\Bigg)^{1/2}\sqrt{\log N_\ell}.
\end{align}
We will show that 
\begin{align}
\sum\limits_{n \leq N_\ell}\frac{|\Omega_{j_0}(n)/\tilde{\sg}_{j_0}(N_\ell) -1|^2}{n} = O\left(\frac{\log N_\ell}{\tilde{\sg}_{j_0}(N_\ell)}\right). \label{eq:Errterm}
\end{align}
After expanding the square on the left-hand side, to prove \eqref{eq:Errterm} it suffices to show that for $j = 0,1,2$,
\begin{align*}
    \sum\limits_{n \leq N_\ell}\frac{(\Omega_{j_0}(n)/\tilde{\sg}_{j_0}(N_\ell))^j}{n} = \log N_\ell + O\left(\frac{\log N_\ell}{\tilde{\sg}_{j_0}(N_\ell)}\right).
\end{align*}
This bound clearly holds when $j = 0$. When $j = 1$, we have
\begin{align*}
    \frac{1}{\tilde{\sg}_{j_0}(N_\ell)}\sum\limits_{n \leq N_\ell}\frac{\Omega_{j_0}(n)}{n} &= \frac{1}{\tilde{\sg}_{j_0}(N_\ell)}\sum\limits_{n \leq N_\ell}\sum\limits_{\substack{mp^k= n, \\ \psi(p) = \omega^{j_0}}}\frac{1}{mp^k} = \frac{1}{\tilde{\sg}_{j_0}(N_\ell)}\sum\limits_{\substack{p^k \leq N_\ell, \\\psi(p) = \omega^{j_0} }}\frac{1}{p^k}\sum\limits_{m \leq N_\ell/p^k}\frac{1}{m}\\
    &= \frac{1}{\tilde{\sg}_{j_0}(N_\ell)}\sum\limits_{\substack{p^k \leq N_\ell, \\\psi(p) = \omega^{j_0} }}\frac{1}{p^k}(\log N_\ell - k\log p + O(1)).
\end{align*}
Using Mertens' theorem, we find that
\begin{align*}
    \sum\limits_{\substack{p^k \leq N_\ell, \\\psi(p) =\omega^{j_0} }}\frac{k\log p}{p^k} &\leq \log N_\ell+\sum\limits_{2 \leq k \leq \log N_\ell}k\sum\limits_{2 \leq n \leq N_\ell^{1/k}}\frac{1}{n^{k-1/2}}\\
    &\ll \log N_\ell+\sum\limits_{2 \leq k \leq \log N_\ell}\frac{k}{k -3/2} \ll \log N_\ell.
\end{align*}
We also note that
$$
\sum_{\ss{p^k \leq N_\ell \\ k \geq 2 \\ \psi(p) = \omega^{j_0}}} \frac{\log N_\ell}{p^k} + \sum_{\ss{p^k \leq N_\ell \\ \psi(p) = \omega^{j_0}}} \frac{1}{p^k} \ll \log N_\ell. 
$$
Hence, we obtain
\begin{align*}
    \frac{1}{\tilde{\sg}_{j_0}(N_\ell)}\sum\limits_{n \leq N_\ell}\frac{\Omega_{j_0}(n)}{n} = \log N_\ell + O\left(\frac{\log N_\ell}{\tilde{\sg}_{j_0}(N_\ell)}\right).
\end{align*}
Finally, when $j = 2$,
\begin{align*}
\frac{1}{\tilde{\sg}_{j_0}(N_\ell)^2}\sum\limits_{n \leq N_\ell}\frac{\Omega_{j_0}(n)^2}{n} 
&= \frac{1}{\tilde{\sg}_{j_0}(N_\ell)^2}\sum\limits_{n \leq N_\ell} \Bigg(\sum\limits_{\substack{p_1^{k_1}|n, \\ \psi(p_1) = \omega^{j_0}}}\frac{1}{n}\Bigg)\Bigg(\sum\limits_{\substack{p_2^{k_2}|n, \\ \psi(p_2) = \omega^{j_0}}}\frac{1}{n}\Bigg)\\
&=\frac{1}{\tilde{\sg}_{j_0}(N_\ell)^2} \sum\limits_{\substack{p_1^{k_1}, p_2^{k_2} \leq N_\ell, \\\psi(p_1)=\psi(p_2) = \omega^{j_0} }}\sum\limits_{\substack{n \leq N_\ell, \\p_1^{k_1}, p_2^{k_2} |n}}\frac{1}{n}\\
&= \frac{1}{\tilde{\sg}_{j_0}(N_\ell)^2}\sum\limits_{\substack{[p_1^{k_1},p_2^{k_2}] \leq N_\ell, \\\psi(p_1)=\psi(p_2) = \omega^{j_0} }}\frac{1}{[p_1^{k_1}, p_2^{k_2}]} \left(\log (N_\ell/[p_1^{k_1},p_2^{k_2}]) + O(1)\right).
\end{align*}
Using Mertens' theorem, we deduce that
\begin{align*}
&\sum_{\ss{[p_1^{k_1},p_2^{k_2}] \leq N_\ell \\ \psi(p_1) = \psi(p_2) = \omega^{j_0}}} \frac{1}{[p_1^{k_1},p_2^{k_2}]} = \sum_{\ss{[p_1,p_2] \leq N_\ell \\ \psi(p_1) = \psi(p_2) = \omega^{j_0}}} \frac{1}{[p_1,p_2]} + O(\tilde{\sg}_{j_0}(N_\ell)) = \sum_{\ss{p_1,p_2 \leq N_\ell \\ \psi(p_1) = \psi(p_2) = \omega^{j_0}}} \frac{1}{[p_1,p_2]} + O(\tilde{\sg}_{j_0}(N_\ell)), \\
&\sum_{\ss{[p_1^{k_1},p_2^{k_2}] \leq N_\ell \\ \psi(p_1) = \psi(p_2) = \omega^{j_0}}} \frac{\log [p_1^{k_1},p_2^{k_2}]}{[p_1^{k_1},p_2^{k_2}]} \ll \left(\sum_{\ss{p_1^{k_1} \leq N_\ell \\ \psi(p_1) = \omega^{j_0}}} \frac{1}{p_1^{k_1}}\right) \left(\sum_{p_2^{k_2} \leq N_\ell} \frac{\log p_2^{k_2}}{p_2^{k_2}}\right) \ll \tilde{\sg}_{j_0}(N_\ell) \log N_\ell.
\end{align*}
Inserting these bounds in the previous estimates, we deduce that
\begin{align*}
\frac{1}{(\tilde{\sg}_{j_0}(N_\ell))^2}\sum\limits_{n \leq N_\ell}\frac{(\Omega_{j_0}(n))^2}{n} 
 &= \frac{\log N_\ell}{(\tilde{\sg}_{j_0}(N_\ell))^2}\sum\limits_{\substack{p_1, p_2 \leq N_\ell, \\\psi(p_1), \psi(p_2) = \omega^{j_0} }}\frac{1}{[p_1, p_2]} + O\left(\frac{\log N_\ell}{\tilde{\sg}_{j_0}(N_\ell)}\right)\\
&= \frac{\log N_\ell}{(\tilde{\sg}_{j_0}(N_\ell))^2}\sum\limits_{\substack{p_1 \\ p_2 \leq N_\ell, p_1 \neq p_2 \\\psi(p_1), \psi(p_2) = \omega^{j_0} }}\frac{1}{p_1p_2} + O\left(\frac{\log N_\ell}{\tilde{\sg}_{j_0}(N_\ell)}\right)\\
&=\log N_\ell + O\left(\frac{\log N_\ell}{\tilde{\sg}_{j_0}(N_\ell)}\right),
\end{align*}
as required. \\
Combining \eqref{eq:Errterm} with \eqref{eq:CSbeforeTK} and combining the result into \eqref{eq:LpsiN}, we find that 
\begin{align*}
\frac{L_{\psi^\ell}(N_\ell)}{\log q}&=\frac{1}{ \tilde{\sg}_{j_0}(N_\ell)\log q} \sum\limits_{\substack{p^k \leq N_\ell, \\ \psi(p) = \omega^{j_0}}}\frac{\psi^\ell(p)^k}{p^k}\sum\limits_{m \leq N_\ell /p^k}\frac{\psi^\ell(m)}{m} + O\Bigg(\frac{\log N_\ell}{(\log q)\sqrt{\tilde{\sg}_{j_0}(N_\ell)}}\Bigg)\\
&= \frac{1}{\tilde{\sg}_{j_0}(N_\ell)\log q } \sum\limits_{\substack{p^k \leq N_\ell, \\ \psi(p) = \omega^{j_0}}}\frac{\omega^{j_0 k\ell}}{p^k} L_{\psi^\ell}(N_\ell/p^k)
+ O\Bigg(\frac{\log N_\ell}{(\log q)\sqrt{\tilde{\sg}_{j_0}(N_\ell)}}\Bigg).
\end{align*}
Since whenever $p^k \leq N_\ell$ we have
\begin{align*}
L_{\psi^\ell}(N_\ell/p^k) = L_{\psi^\ell}(N_\ell) + O\left(\sum_{N_\ell/p^k < m \leq N_\ell} \frac{1}{m}\right) = L_{\psi^\ell}(N_\ell) + O(k\log p),
\end{align*}
we obtain using Mertens' theorem again that
\begin{align*}
\frac{L_{\psi^\ell}(N_\ell)}{\log q} 
&= \frac{1}{\tilde{\sg}_{j_0}(N_\ell)\log q } \sum\limits_{\substack{p^k \leq N_\ell, \\ \psi(p) = \omega^{j_0}}}\frac{\omega^{j_0k\ell}}{p^k}\Big(L_{\psi^\ell}(N_\ell) + O(k\log p)\Big) + O\Bigg(\frac{\log N_\ell}{(\log q)\sqrt{\tilde{\sg}_{j_0}(N_\ell)}}\Bigg) \\
&= \frac{\omega^{j_0\ell}L_{\psi^\ell}(N_\ell)}{\tilde{\sg}_{j_0}(N_\ell)\log q } \left(\sum\limits_{\substack{p \leq N_\ell, \\ \psi(p) = \omega^{j_0}}}\frac{1}{p}\right) + O\left(\frac{\log N_\ell}{(\log q)\sqrt{\tilde{\sg}_{j_0}(N_\ell)}}\right) \\
&= \omega^{j_0\ell}\frac{L_{\psi^\ell}(N_\ell)}{\log q}  + O\left(\frac{\log N_\ell}{(\log q)\sqrt{\tilde{\sg}_{j_0}(N_\ell)}}\right).
\end{align*}
Rearranging this expression and using the bound $|1-\omega^{j_0\ell}| \gg \|j_0\ell/r\|$, we deduce that
$$
\frac{|L_{\psi^\ell}(N_\ell)|}{\log q}  \ll \left\|\frac{j_0\ell}{r}\right\|^{-1} \frac{\log N_\ell}{(\log q) \sqrt{\tilde{\sg}_{j_0}(N_\ell)}},
$$
as claimed.
\end{proof}
In analogy to Proposition \ref{prop:discBound}, the number of $\ell$ for which  $\Big\|\frac{j_0g\ell}{d}\Big\|$ can also be estimated effectively. 
\begin{prop}
With the above notation there is an absolute constant $C > 0$ such that if
$$
\theta_0 := C\left(\frac{(mg)^2\eta^{-1}\log(1/\e)}{\log\log r}\right)^{1/2}.
$$
then at least one of the following is true:
\begin{enumerate}[(i)]
\item $|M(\chi)| \leq \e\sqrt{q} \log q$, and
\item $(j_0,r) \leq \theta_0 r$.
\end{enumerate}
Moreover, in the second case we find that for any $\theta \in [\theta_0,1/2]$,
\begin{align*}
|\{1 \leq \ell \leq r : \|j_0 \ell/r\| \leq \theta\}| \ll \theta r.  
\end{align*}
\end{prop}
\begin{proof}
Assume (i) fails, so $|M(\chi)| \geq \e \sqrt{q}\log q$, i.e., $1 \in \mc{L}_d(\e)$. By Proposition 3.2, $H \cong \Z/d\Z$ implies $g = 1$, and 
\begin{align*}
\mb{D}(\psi, 1; q)^2 \ll m^2\log (1/\e)
\end{align*}
where $\psi = \chi \bar{\xi}$.
Arguing as in \eqref{eq:j0SigmaBd}, we have
\begin{align*}
m^2\log (1/\e) \gg \eta \Big\|\frac{j_0}{d}\Big\|^2\Sigma_{\psi}(q). 
\end{align*}
Following the same argument as in the proof of Proposition \ref{prop:discBound}, now gives the claim.
\end{proof}

\begin{proof}[Proof of Theorem 1.3]
Let $\e > 0$ be a parameter to be chosen shortly, and suppose throughout that $|M(\chi)| \geq \e \sqrt{q}\log q$. Assume for the sake of contradiction that $\mc{L}_{d}(\e) \geq \e d$.
\\
By Proposition 3.2, there are integers $1 \leq m \leq \e^{-2}$ and $1 \leq g \leq \e^{-1}$ and a primitive Dirichlet character $\xi \pmod{k}$ with $1 \leq k \leq \e^{-3m}$ such that if $\psi := \chi^g \bar{\xi}$ then
\begin{align*}
\max_{1 \leq \ell \leq d/g}\mb{D}(\psi^\ell, 1; q)^2 \ll m^2\log (1/\e).
\end{align*}
We now set $\e = \Sigma_{\psi}(q)^{-1/8}$. Fix also $\eta \in (0, 1)$ small such that $\eta > \e^{1/6}$, and let $1 \leq j_0\leq d-1$ be an index satisfying 
\begin{align*}
\tilde{\sg}_{j_0}(q) = \max_{1 \leq j \leq d-1}\tilde{\sg}_{j}(q).
\end{align*}
Set $r := d/g$ as previously. If $\sg_{j_0}(q) \leq \frac{\eta}{g}\Sigma_{\psi}(q)$, then by Proposition 5.1 there is $1 \leq \ell \leq r$ such that
\begin{align*}
\mb{D}(\psi^\ell, 1; q)^2 \geq \frac{1}{2}\Sigma_{\psi}(q).
\end{align*}
Using Proposition 2.5 and $m \leq \e^{-2}$, we have
\begin{align*}
\log \log r \ll \Sigma_{\psi}(q) \ll m^2\log(1/\e) \leq \e^{-4}\log(1/\e).
\end{align*}
Since $\e = \Sigma_{\psi}(q)^{-1/8}$ we obtain a contradiction as soon as $d$ is large enough.
\\
Next, suppose $\sg_{j_0}(q) > \frac{\eta}{g}\Sigma_{\psi}(q)$. Let $\theta \in [\theta_0, 1/2]$ be a parameter that satisfies $\theta > \e^{1/6}$. Since $|M(\chi)| \geq \e\sqrt{q}\log q$, by Proposition 5.3, we have
\begin{align*}
|\{1 \leq \ell \leq d: \|j_0\ell/d\| \leq \theta \}| \leq \theta d.
\end{align*}
By the same argument as in the proof of Theorem 1.2, there must exist $g\ell \in \mc{L}_{d}(\e)$ such that $\|j_0g\ell/d\| > \theta/ m$. Moreover, as noted previously we have $N_\ell \gg q^\e$, so by Mertens' theorem
\begin{align*}
\sg_{j_0}(N_\ell) = \sg_{j_0}(q)  -\sum\limits_{\substack{N_\ell < p \leq q, \\ \psi(p) = \omega^{j_0}}} \frac{1}{p} \geq  \frac{\eta}{g}\Sigma_{\psi}(q) - \log (1/\e) + O(1) \geq \frac{\eta}{2g}\Sigma_{\psi}(q).
\end{align*}
Thus, as $\theta,\eta > \e^{1/6}, m \leq \e^{-2}$ and $g \leq \e^{-1}$,
\begin{align*}
\e \leq \frac{1}{\log q} |L_{\psi^\ell}(N_\ell)| \ll \Bigg(\frac{m^2g}{\theta^2\eta\Sigma_{\psi}(q)}\Bigg)^{1/2} \ll \frac{1}{\sqrt{\e^{11/2} \Sigma_{\psi}(q)}}
\end{align*}
Given that $\Sigma_{\psi}(q) = \e^{-8}$ we again obtain a contradiction. Therefore, we conclude that $|\mc{L}_d(\e)| \leq \e d$. Applying Proposition \ref{prop:LowBdSigma} together with the lower bound $r \geq \e d$, we have
$$
\e \ll (\log\log r)^{-1/8} \ll (\log\log d)^{-1/8},
$$
and therefore
\begin{align*}
\frac{1}{d}\sum\limits_{1 \leq \ell \leq d}|M(\chi^\ell)| \ll \Big(\e + \frac{|\mc{L}_d(\e)|}{d}\Big)\left(\sqrt{q}\log q\right) \ll \frac{\sqrt{q}\log q}{(\log \log d)^{1/8}} 
\end{align*}
as desired.
\end{proof}

\bibliographystyle{plain}
\bibliography{HighOrdGenBib}

\begin{thebibliography}{10}

\bibitem{bajnok}
B.~Bajnok.
\newblock On the maximum size of a (k, l)-sum-free subset of an abelian group.
\newblock {\em International Journal of Number Theory}, 5(06):953--971, 2009.

\bibitem{Gold}
L.~Goldmakher.
\newblock Multiplicative mimicry and improvements of the {P}\'olya-{V}inogradov
  inequality.
\newblock {\em Algebra Number Theory}, 6(1):123--163, 2012.

\bibitem{GraMan}
A.~Granville and A.P. Mangerel.
\newblock Three conjectures about character sums.
\newblock {\em Math. Zeit.}, 305:49:34 pp., 2023.

\bibitem{GSDecay}
A.~Granville and K.~Soundararajan.
\newblock Decay of mean values of multiplicative functions.
\newblock {\em Can. J. Math.}, 55(6):1191--1230, 2003.

\bibitem{GSPret}
A.~Granville and K.~Soundararajan.
\newblock Large character sums: pretentious characters and the
  {P}\'{o}lya-{V}inogradov inequality.
\newblock {\em J. Amer. Math. Soc.}, 20(2):357--384, 2007.

\bibitem{Hall}
R.R. Hall.
\newblock Halving an estimate obtained from {S}elberg's upper bound sieve.
\newblock {\em Acta Arith.}, 25:347--351, 1974.

\bibitem{HT}
R.R. Hall and G.~Tenenbaum.
\newblock Effective mean value estimates for complex multiplicative functions.
\newblock {\em Math. Proc. Camb. Philos. Soc.}, 110:337--351, 1991.

\bibitem{ould}
Y.~o. Hamidoune and A.~Plagne.
\newblock A new critical pair theorem applied to sum-free sets in abelian
  groups.
\newblock {\em Commentarii Mathematici Helvetici}, 79:183--207, 2004.

\bibitem{Hil}
A.~Hildebrand.
\newblock Quantitative mean value theorems for nonnegative multiplicative
  functions {II}.
\newblock {\em Acta Arith.}, 48:209--260, 1987.

\bibitem{LamMan}
Y.~Lamzouri and A.P. Mangerel.
\newblock Large odd order character sums and improvements to the
  {P}\'{o}lya-{V}inogradov inequality.
\newblock {\em Trans. Amer. Math. Soc.}, 375:3759--3793, 2022.

\bibitem{ManInhom}
A.P. Mangerel.
\newblock Gap problems for integer-valued multiplicative functions.
\newblock arXiv:2311.11636 [math.NT].

\bibitem{ManHighOrd}
A.P. Mangerel.
\newblock Large sums of high order characters.
\newblock {\em J. Lond. Math. Soc.}, 109(1), 2024.

\bibitem{Ruz}
I.Z. Ruzsa.
\newblock On the concentration of additive functions.
\newblock {\em Acta Math. Acad. Sci. Hung.}, 36:215--232, 1980.

\bibitem{Ten}
G.~Tenenbaum.
\newblock {\em Introduction to Analytic and Probabilistic Number Theory}.
\newblock Graduate Studies in Mathematics vol. 163, AMS Publications,
  Providence, RI, 2015.

\end{thebibliography}
\end{document}